\documentclass[11pt]{amsart}
\usepackage{amsmath,amsthm,amscd,amsfonts,amssymb,epic,eepic,bbm, graphicx, youngtab,txfonts, ytableau}
\usepackage{tikz,color,graphicx,tikz-cd,xcolor,tkz-euclide}
\usepackage[draft]{hyperref}
\usepackage{enumerate}
\usepackage{array}
\usepackage{mathtools}

\usepackage{kotex-logo}

\usepackage[T1]{fontenc}
\PassOptionsToPackage{no-math}{fontspec}
\usepackage{iftex}

\ifPDFTeX
\usepackage[finemath]{kotex}
\usepackage{dhucs-nanumfont}
\fi

\ifXeTeX
\usepackage[unfonts]{kotex}
\xetexkofontregime{latin}[puncts=prevfont, colons=prevfont, cjksymbols=hangul]

\defaultfontfeatures+{
Ligatures          = TeX,
ItalicFont         = *,
ItalicFeatures     = {FakeSlant=.17}, 
BoldItalicFeatures = {FakeSlant=.17},
}

\setmainhangulfont{NanumMyeongjo}
\setsanshangulfont{NanumGothic}
\fi
\numberwithin{equation}{section}
\allowdisplaybreaks

\catcode`\@=11
\@namedef{subjclassname@2020}{%  
  \textup{2020} Mathematics Subject Classification}
\catcode`\@=12

\newtheorem{thm}{Theorem}[section]
\newtheorem{cor}[thm]{Corollary}

\newtheorem{lem}[thm]{Lemma}

\newtheorem{prop}[thm]{Proposition}

\theoremstyle{definition}

\theoremstyle{remark}

\DeclareMathOperator\prmx{fdes}
\DeclareMathOperator\rank{rank}
\DeclareMathOperator\vox{vox}
\DeclareMathOperator\height{ht}
\DeclareMathOperator\block{block}

\def\Z{\mathbb{Z}}

\newcommand\set[1]{\left\{#1\right\}}

\def\IS{\mathcal{IS}}
\def\UVD{\mathcal{UVD}}
\def\SP{\mathcal{SP}}
\newcommand\LFp[1]{\mathcal{LF}_{#1}}

\newcommand\LFa{\mathcal{A}}
\newcommand\LFb{\mathcal{B}}

\newcommand\wtd[1]{\widetilde{#1}}

\title%[]%
{On 102-avoiding Inversion Sequences}

\author{JiSun Huh}
\address[JiSun Huh]{Department of Mathematics, University of Seoul, Seoul, Republic of Korea}
\email{hyunyjia@yonsei.ac.kr}
\thanks{}

\author{Sangwook Kim$^\dag$}
\address[Sangwook Kim]{Department of Mathematics, Chonnam National University, Gwangju, 61186, South Korea}
\email{swkim.math@chonnam.ac.kr}
\thanks{\dag Corresponding author}
\thanks{Sangwook Kim was partially supported by 
the National Research Foundation of Korea (NRF) grant funded by the Korea government (MSIT) (No.\ 2021R1F1A1062356).
}

\author{Seunghyun Seo}
\address[Seunghyun Seo]{Department of Mathematics Education, Kangwon National University, Chuncheon, 24341, South Korea}
\email{shyunseo@kangwon.ac.kr}
\thanks{}

\author{Heesung Shin}
\address[Heesung Shin]{Department of Mathematics, Inha University, 100 Inharo, Michuhol, Incheon, 22212, South Korea}
\email{shin@inha.ac.kr}
\thanks{Heesung Shin was supported by the National Research Foundation of Korea (NRF) grant funded by the Korea government (MSIT) (No. 2017R1C1B2008269).}

\begin{document}

\begin{abstract}
  In this article, we provide a bijection between the set of inversion sequences avoiding the pattern 102 
  and the set of 2-Schr\"{o}der paths having neither peaks nor valleys and ending with a diagonal step.
  To achieve this, we introduce two intermediate objects, called UVD paths and labeled $F$-paths, 
  and establish bijections among all four families. For each of these combinatorial objects, 
  we define a natural statistic and enumerate the corresponding structures with respect to this statistic. 
  In addition, we study inversion sequences avoiding 102 and another pattern of length 3, 
  providing refined enumerations according to the same statistic.
\end{abstract}

\maketitle
    
%\todototoc
%\listoftodos

%\tableofcontents

%%%%%%%%%%%%%%%%
\section{Introduction}

An integer sequence \( e=(e_1, e_2, \dots, e_n) \) is called an \emph{inversion sequence of length \( n \)} if \( 0 \leq e_i <i \) for 
all \( i \in [n] \).
While pattern avoidance in permutations has been extensively studied (see Kitaev~\cite{Kitaev11} for the survey), 
the study of pattern avoidance in inversion sequences is a more recent development,
initiated independently by Corteel et al.~\cite{CMSW16}, and 
Mansour and Shattuck~\cite{MS15}.
They studied inversion sequences avoiding patterns of length $3$.
Following the initial work, 
Martinez and Savage~\cite{MartinezSavage18} reframed the notion of length-three 
pattern from a word of length $3$ to a triple of binary relations.
Yan and Lin~\cite{YanLin20} further investigated Wilf-equivalences among pairs of such patterns.
Recently, Testart~\cite{Testart24} completed the enumeration of inversion sequences avoiding one or two patterns of length \( 3 \). 
Meanwhile, Hong and Li~\cite{HongLi22} nearly completed the Wilf classification for patterns of length four.

Mansour and Shattuck~\cite{MS15} found that the generating function $A(x)$ for the number of $102$-avoiding inversion sequences satisfies functional equation
$$A(x) = 1+ \left( x-x^2 \right) A(x)^3.$$
Interestingly, Seo and Shin~\cite{SS23} recently showed that the generating function for the number of 
$2$-Schr\"{o}der paths 
having neither peaks nor valleys 
and ending with a diagonal step satisfies the same equation,
demonstrating that the two sets are equinumerous. They posed the open problem of finding a bijection between the two sets.
The number of inversion sequences avoiding $102$ and another pattern of length $3$ was studied by
several authors~\cite{CMSW16, KMY24, Testart24, YanLin20}.

In this paper, we resolve this problem by introducing two new objects:
UVD paths and labeled $F$-paths.
We construct bijections among the set of UVD paths, the set of labeled $F$-paths, and the set of $102$-avoiding inversion sequences.
As a consequence, we obtain a bijection between $102$-avoiding inversion sequences and
$2$-Schr\"{o}der paths having neither peaks nor valleys and ending with a diagonal step,
thus answering the question of Seo and Shin~\cite{SS23}.

We further define a statistic called \emph{rank} on $102$-avoiding inversion sequences extending the notion introduced for \( (102,101) \)-avoiding sequences in~\cite{HKSS24}. Analogous statistics are also defined
for UVD paths and labeled $F$-paths. 
We then enumerate \( 102 \)-avoiding inversion sequences with a fixed rank and study the doubly-avoiding cases 
where a second pattern \( \tau \in \{001, 011, 012, 021, 110, 120, 201, 210\} \) is also avoided.
The cases involving patterns \( 000, 010, 100 \) remain open.

The rest of the paper is organized as follows.
Section~\ref{sec:preliminaries} introduces definitions and background.
In Section~\ref{sec:bijections}, we construct bijections among $102$-avoiding inversion sequences, UVD paths, and labeled $F$-paths.
Section~\ref{sec:102-with-rank} provides refined enumerations based on the rank statistic, 
while Section~\ref{sec:102-tau-with-rank} explores the doubly-avoiding cases for various patterns \( \tau \).

%===============================================================
\section{Preliminaries}
\label{sec:preliminaries}

In this section, we provide definitions for the objects we discuss in this article.
Let $n$ be a positive integer.

%----------------------------------------------------------------
\subsection{Inversion sequences}
An integer sequence $e=(e_1, e_2, \dots, e_n)$ is called an \emph{inversion sequence of length $n$}
if $0 \leq e_j \leq j-1$ for all $j\in [n]:=\{ 1,2,\dots,n \}$.
We denote the largest 
integer in $e$ by $\max(e)$.
We also denote by \( \prmx(e) \) 
the position $p$ of the first descent, i.e.,
\begin{align*}
  e_1 \leq e_2\leq \dots \leq e_p > e_{p+1}
\end{align*}
with $e_{n+1} = -1$.

Pattern avoidance in inversion sequences can be defined in a similar way 
to that in permutations.
Given a word \( w \in \set{0, 1, \dots, k-1}^k\),
let the word obtained by replacing the $i$-th smallest entry in $w$ with $i-1$
be called the \emph{reduction} of $w$.
We say that an inversion sequence \( e=(e_1, e_2, \dots, e_n) \) 
\emph{contains} the pattern $w$
if there exist some indices $i_1 < i_2 < \dots < i_k$ 
such that the reduction of $e_{i_1} e_{i_2} \dots e_{i_k}$ is $w$.
Otherwise, $e$ is said to \emph{avoid} the pattern $w$.
Let $\IS_n$ denote the set of inversion sequences of length $n$ and 
$\IS_n(w_1, \dots, w_r)$ denote the set of inversion sequences of length $n$ avoiding all the patterns $w_1, \dots, w_r$.

For 
$e=(e_1, e_2, \dots, e_n) \in \IS_n(102)$,
if $p \!= \prmx(e)$,
then 
$e_p = \max(e)$.
In this case, we define 
\begin{align*}
\rank(e) 
:= \prmx(e) - \max(e) - 1
= (p - 1) - e_p \geq 0.
\end{align*}

%------------------------------------------------------------------

\subsection{$2$-Schr\"{o}der paths and UVD paths}

A \emph{$2$-Schr\"{o}der path of semilength $n$} is a lattice path
from $(0, 0)$ to $(n, 2n)$ 
that does not go below the line $y=2x$ and consists of 
north steps $N = (0, 1)$, east steps $E = (1, 0)$, 
and diagonal steps $H= (1, 1)$.

Let $\SP_n$ denote the set of $2$-Schr\"{o}der paths of semilength $n$ 
having neither peaks $NE$ nor valleys $EN$ 
and ending with a diagonal step $H = (1, 1)$.
Note that $\SP_n$ is the set $\mathcal{DP}_{n,2n}^{(2)}(NE, EN)$ introduced in \cite{SS23}.
A diagonal step that touches the line $y=2x$ is called a \emph{return} 
in a $2$-Schr\"{o}der path.
The number of returns of a $2$-Schr\"{o}der path $P$ is denoted by $\block(P)$.

Define a \emph{UVD path $S$ of semilength $n$} to be a lattice path from $(0,0)$ to $(2n,0)$ that stays weakly above the $x$-axis, consists of up steps $u=(1,1)$, down steps $d=(1,-1)$, and vertical steps $v=(0,-2)$, avoids the consecutive patterns $uv$ and $vu$, and ends with a $d$ step.
Note that a UVD path  of semilength $n$ can be expressed as a word 
$S=s_1 s_2 \ldots s_{2n+r}$	
consisting of $n+r$ up steps, $n-r$ down steps, and $r$ vertical steps, for some nonnegative integer $r$.

Let $\UVD_n$ denote the set of UVD paths of semilength $n$. By the linear transformation $M:\mathbb{R}^2\to\mathbb{R}^2$ induced by the matrix 
$M = \begin{psmallmatrix}
  0 & 1 \\ 
  -2 & 1
\end{psmallmatrix}$, a path $P$ in $\SP_n$ corresponds to a UVD path $S$ in $\UVD_n$, which 
is
denoted by $S=MP$. Therefore, we can regard $M$ as a map from $\SP_n$ to $\UVD_n$. 
In fact, $M$ is a bijection that sends 
$N$ to $u$, $E$ to $v$, and $H$ to $d$.

Given a UVD path $S$, a down step in $S$ that touches the $x$-axis is called a 
\emph{return}.
Let $\vox(S)$ denote the number of valleys $du$ on the $x$-axis of $S$,
and set $\vox(\emptyset):=-1$ by convention.
Then the number of returns of $S$ is clearly equal to $\vox(S)+1$. 
 
Therefore, if $S=MP$, then
$$
\vox(S)=\block(P)-1,
$$
i.e., $M$ is a bijection from $\SP_n$ to $\UVD_n$ with 
$\block(P)=\vox(MP)+1$.

%---------------------------------------------------------------------

\subsection{Labeled $F$-paths}

In \cite{HKSS24}, Huh et al. introduced \emph{F-paths} and constructed a bijection between 
$(102, 101)$-avoiding
inversion sequences and $F$-paths. 
An \emph{$F$-path of length $n$} is a lattice path in $\Z^2$
that starts at the origin and does not go below the line $y=x$
as a sequence of lattice points
$$(x_0, y_0), (x_1, y_1), \dots, (x_{\ell}, y_{\ell})$$
of which every step $(x_{j} - x_{j-1}, y_{j}-y_{j-1})$, denoted by $s_j$, 
is 
in the set
$$F:=\set{(a, 1): a \geq 0} \cup \set{(a, b): a\geq 1, b \leq 0},$$
for $j=1, 2, \ldots, n$.

Here, we extend the family of $F$-paths
by assigning a label to each step, in order to accommodate
the superset $\IS_n(102)$ of $\IS_n(102,101)$.
A \emph{labeled $F$-path} is an $F$-path 
where every 
step $(a, 1)$ is
assigned a label $(a; 1)$ and 
every other step $(a, b)$ with $b\leq 0$ is assigned a label 
$(a; b_1, \dots, b_k)$
for some nonpositive integers $b_1, \dots, b_k$ with $k \ge 1$ such that $b_1 + \cdots + b_k = b$.
We say that a step with a label $(a; b_1, \dots, b_k)$ has the \emph{semilength} $k$.
Define the semilength of a labeled $F$-path by the sum of the semilengths of its steps.
In Figure~\ref{fig:F},
the labeled $F$-path $Q$ has exactly $19$ steps, 
but the semilength of $Q$ is equal to $24$, 
since the semilengths of the last two steps of $Q$ are 3 and 4.
Let $\LFp{n}$ be the set of labeled $F$-paths of semilength $n$.

Given a labeled $F$-path 
$Q = (0, 0), (x_1, y_1), \dots, (x_\ell, y_\ell),$
define the \emph{height} of $Q$
by the value $y_\ell-x_\ell$ of the last lattice point $(x_\ell, y_\ell)$ of $Q$
and denote it by $\height(Q)$.

%----------------------------------------------------
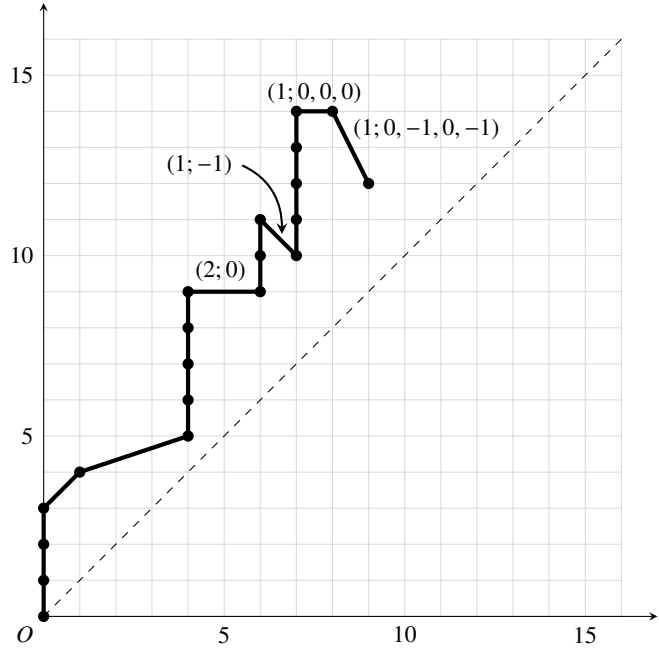
\begin{figure}[t]
\[
  \begin{tikzpicture}[scale=0.48]
    \pgfmathsetmacro{\n}{16}
    \pgfmathsetmacro{\m}{\n+1}
    \draw[gray!30, ultra thin] (0,0) grid (\n,\n);
    \draw [-stealth] (0,0) -- (\m,0);
    \draw [-stealth] (0,0) -- (0,\m);
    \draw [dashed] (0,0) -- (\n,\n);

    \coordinate (0) at (0,0);
    \coordinate (1) at (0,1);
    \coordinate (2) at (0,2);
    \coordinate (3) at (0,3);
    \coordinate (4) at (1,4);
    \coordinate (a) at (2.5,4.5);
    \coordinate (5) at (4,5);
    \coordinate (6) at (4,6);
    \coordinate (7) at (4,7);
    \coordinate (8) at (4,8);
    \coordinate (9) at (4,9);
    \coordinate (b) at (4.9,9);
    \coordinate (10) at (6,9);
    \coordinate (11) at (6,10);
    \coordinate (12) at (6,11);
    \coordinate (c1) at (5.5,12.5);
    \coordinate (c2) at (6.6,12);
    \coordinate (c3) at (6.6,11);
    \coordinate (c4) at (6.6,10.6);
    \coordinate (13) at (7,10);
    \coordinate (14) at (7,11);
    \coordinate (15) at (7,12);
    \coordinate (16) at (7,13);
    \coordinate (17) at (7,14);
    \coordinate (d) at (7.5,14.5);
    \coordinate (18) at (8,14);
    \coordinate (e) at (8.3,13.5);
    \coordinate (19) at (9,12);

    \foreach \i in {0,...,19}{
      \filldraw (\i) circle (4pt);
    }

    \foreach \i in {1,...,19}{
    \pgfmathsetmacro{\j}{\i-1}
    \draw[ultra thick] (\j) -- (\i);
    }

    \node at (0, 0) [below left] {\footnotesize $O$};
    \foreach \i in {5,10,...,17}{
      \node at (\i, 0) [below] {\footnotesize$\i$};
      \node at (0,\i) [left] {\footnotesize$\i$};
    }
    \node at (b) [above] {\footnotesize$(2;0)$};
    \node at (c1) [left] {\footnotesize$(1;-1)$};
    \draw [thick, -stealth] (c1) .. controls (c2) and (c3) .. (c4);
    \node at (d) [] {\footnotesize$(1;0,0,0)$};
    \node at (e) [right] {\footnotesize$(1;0, -1, 0, -1)$};
  \end{tikzpicture}
\]
\caption{
  A labeled $F$-path $Q$ of semilength $24$ and height $3$,
  with the labels $(a;1)$ are omitted for brevity.
}
\label{fig:F}
\end{figure}
%----------------------------------------------------

%==================================================
\section{Bijections}
\label{sec:bijections}

In this section, we give bijections among the sets $\mathcal{LF}_n$, 
$\IS_{n+1}(102)$, and $\mathcal{UVD}_{n+1}$.

%-----------------------------------------------------
\subsection{A bijection $\phi$ from $\mathcal{LF}_{n}$ to $\IS_{n+1}(102)$}
\label{sec:phi}

We construct a map 
from labeled $F$-paths to $102$-avoiding inversion sequences.
Define 
$$\phi : \bigcup_{n\geq 0} \mathcal{LF}_{n}
\to \bigcup_{n\geq 0} \IS_{n+1}(102)$$
recursively, with the property that $\height(Q) = \rank(\phi(Q))$ as follows:

For the initial case $n=0$, we define $\phi((x_0, y_0)) := (0)$, ensuring that $\height((x_0, y_0)) = \rank((0)) = 0$.
For the recursive step $n\geq 1$, 
consider a labeled $F$-path $Q = (x_0, y_0), (x_1, y_1), \dots, (x_\ell, y_\ell)$ in $\LFp{n}$ 
with $\ell \geq 1$. 
Let 
\[
\hat{Q}:=(x_0, y_0), (x_1, y_1), \dots, (x_{\ell-1}, y_{\ell-1})
\in \LFp{n-k},
\]
where $k$ is the semilength of the last step of $Q$ with $0 < k \leq n$.
We define
\[
\hat{e} := (\hat{e}_1, \hat{e}_2, \dots, \hat{e}_{n-k+1})
= \phi(\hat{Q})
\in \IS_{n-k+1}(102)
\]
so that $\height(\hat{Q}) = \rank(\hat{e})$ and $\hat{e}_{\hat{p}} = \max(\hat{e})$,
where $\hat{p}=\prmx(\hat{e})$.
Depending on the label of the last step $(a,b)$ of $Q$, we define $\phi(Q):=e$ differently.
Let $\max(\hat{e}) + a=m$.

\begin{enumerate}[(1)]
  \item If the label of the last step of $Q$ is $(a; 1)$ 
  with $a\geq 0$, then $b=k=1$ and 
  $\hat{e} = (\hat{e}_1, \hat{e}_2, \dots, \hat{e}_{n})$. 
  We define $e$ by inserting an entry $m$
  into the sequence $\hat{e}$ 
  after the entry $\hat{e}_{\hat{p}}$, that is,
  $$e := (\hat{e}_1, \hat{e}_2, \dots, \hat{e}_{\hat{p}}, m, \hat{e}_{\hat{p}+1}, \dots, \hat{e}_{n}).$$
  It is straightforward to verify that $e$ avoids the pattern $102$.
  Moreover, we have 
  \begin{align*}
  \max(e) &= \max(\hat{e}) + a = m,\\
  \prmx(e)&= \prmx(\hat{e}) + 1 = \hat{p}+1,\\ 
  \rank(e) &= \rank(\hat{e}) + (1-a) = \height(\hat{Q}) + (1-a) = \height(Q).
  \end{align*}

  \item If the label of the last step $(a,b)$ of $Q$ is $(a; b_1, b_2, \dots, b_k)$ 
  with $a\geq 1$ and $b_1, \dots, b_{k} \leq 0$ so that $b\leq 0$, 
  we define $e$  
  by inserting $k$ instances of $m$ into the sequence 
  $\hat{e} = (\hat{e}_1, \hat{e}_2, \dots, \hat{e}_{n-k+1})$, 
  after the entries $\hat{e}_{j_1}, \hat{e}_{j_2}, \ldots, \hat{e}_{j_k}$,
  where 
\[
  j_1 = \hat{p}+b_1 + b_2 + \dots + b_k - 1=\hat{p}+b-1
 \]
 and
 \[
  j_i = \hat{p} + b_i + b_{i+1} + \dots + b_k \quad \text{for $i =2, \dots, k$}.
  \]
  Since $j_1< j_2 \leq j_3 \leq \dots \leq j_k$ the sequence $e$ can be formally written as
  \[
  e := (\hat{e}_1, \hat{e}_2, \dots, 
  \hat{e}_{j_1}, m, \hat{e}_{j_1+1}, \dots, 
  \hat{e}_{j_2}, m, \hat{e}_{j_2+1}, \dots, 
  \hat{e}_{j_k}, m, \hat{e}_{j_k+1}, \dots, 
  \hat{e}_{n-k+1}).
  \]
  If the values of some $j_i$'s are the same, then $m$ is inserted multiple times at the same position after the entry $\hat{e}_{j_i}$.
  It is straightforward to verify that $e \in \IS_{n+1}(102)$
  with 
  \begin{align*}
  \max(e) &= \max(\hat{e}) + a = m,\\
  \prmx(e) &=  j_1 + 1= \prmx(\hat{e}) + b,\\ 
   \rank(e) &= \rank(\hat{e}) + (b -a) 
  = \height(\hat{Q}) + (b -a) = \height(Q).
  \end{align*}
\end{enumerate}
 Thus, $\phi(Q)$ can be defined by $e$ in $\IS_{n+1}(102)$ in both cases. 
 
Let us apply the map $\phi$
to the labeled $F$-path $Q$ of semilength $24$ in Figure~\ref{fig:F} 
as an example, 
in order to obtain a $102$-avoiding inversion sequence $\phi(Q)$.
Let $Q = (x_0, y_0), (x_1, y_1), \dots, (x_{19}, y_{19}) \in \mathcal{LF}_{24}$, and 
denote $Q^{(j)} = (x_0, y_0), (x_1, y_1), \dots, (x_j, y_j)$ for $j=0,\dots,19$.
It is clear that $Q^{(0)}=(x_0, y_0)$ 
and $Q^{(19)}$ represents the entire path $Q$.
Defining $e^{(j)} = \phi(Q^{(j)})$, 
we obtain the following inversion sequences for some values of $j$: 

\begin{align*}
% e^{(0)} &= 0 &\\
e^{(3)} &= 000\dot0, &
e^{(12)} &= 00001444466\dot64, \\
e^{(4)} &= 0000\dot1, &
e^{(13)} &= 0000144446\dot7664, \\
e^{(5)} &= 00001\dot4, &
e^{(17)} &= 00001444467777\dot7664, \\
e^{(9)} &= 000014444\dot4, &
e^{(18)} &= 00001444467777\dot8788664, \\
e^{(10)} &= 000014444\dot64, &
e^{(19)} &= 000014444677\dot9797998788664 = \phi(Q).
\end{align*}
Here, we write
$e_1 e_2 \dots e_{\ell}$ to denote the sequence $(e_1, e_2, \ldots, e_\ell)$, and
a dot placed above a number indicates the value of $\prmx(e^{(j)})$.

\begin{thm}
\label{thm:IS}
The map 
$$\phi : \bigcup_{n\geq 0} \LFp{n}
\to \bigcup_{n\geq 0} \IS_{n+1}(102)$$
is a bijection with $\height(Q) = \rank(\phi(Q))$.
\end{thm}

\begin{proof}
Note that for a given $Q=(x_0,y_0),(x_1,y_1),\dots, (x_{\ell},y_{\ell})$ in $\LFp{n}$,
it is straightforward to verify that \( x_\ell = \max(\phi(Q)) \)
and \( y_\ell = \prmx(\phi(Q))-1 \). 
To prove the bijectivity of $\phi$, it suffices to show that, given
a $102$-avoiding inversion sequence $e$, 
 one can uniquely reconstruct a labeled $F$-path $Q$ such that $e=\phi(Q)$, by determining
the sequence $\hat{e}=\phi(\hat{Q})$ and the last step
$(a,b)=(x_{\ell}-x_{\ell-1},y_{\ell}-y_{\ell-1})=(\max(e)-\max(\hat{e}), \prmx(e)-\prmx(\hat{e}))$
of $Q$, along with its label.

For $n\geq 1$, let $e$ be a $102$-avoiding inversion sequence in $\IS_{n+1}(102)$ with
$\prmx(e)=p$ and $\max(e)=m$. 
From the definition of $\prmx(e)$, it follows that $e_{p-1} \leq m$ and $e_{p+1} < m$, 
where we adopt the convention $e_{n+2} = -1$.
Hence, we consider the following two cases:

\begin{enumerate}
\item
If $e_{p+1} < e_{p-1} \leq m$, then 
we must have $e_{j} \leq e_{p-1} \leq m$ for all $j$ with $p+2 \leq j \leq n+1$;
otherwise, $e_{p-1} e_{p+1} e_j$ would form the pattern $102$.
Hence, the sequence 
$$\hat{e} = (e_1, e_2, \dots, e_{p-1}, e_{p+1}, \dots, e_{n+1})$$
belongs to $\IS_{n}(102)$ 
with $\max(\hat{e}) = e_{p-1}$ and $\prmx(\hat{e}) = \prmx(e)-1$.
Thus, we recover $(a,b)=(a, 1)$ with label $(a; 1)$, 
where $$a = \max(e) - \max(\hat{e})\geq 0.$$

\item If $e_{p-1} \leq e_{p+1} < m$,  
let the indices of the entries equal to $m$ in $e$
be denoted by $i_1, i_2, \dots, i_k$, satisfying 
\[
p = i_1 \leq i_2 - 2 \leq i_3 - 3 \leq \dots \leq i_k - k,
\]
for some $k\geq 1$.
By removing all the entries equal to $m$ from $e$, 
we obtain the sequence 
$$\hat{e} = (e_1, e_2, \dots, e_{i_1-1}, e_{i_1+1}, \dots, e_{i_2-1}, e_{i_2+1}, \dots \dots, e_{i_k-1}, e_{i_k+1}, \dots, e_{n+1})$$
which belongs to $\IS_{n-k+1}(102)$. Let $\prmx(\hat{e})=\hat{p}$.
Then we recover $(a,b)$ with label $(a;b_1, b_2, \dots, b_k)$, where 
$a=\max(e)-\max(\hat{e})\geq 1$ and $b=b_1+b_2+\cdots+b_k\leq 0$ as follows:
We consider two subcases:
\begin{itemize}
\item If $k=1$, let $b_1 = p - \hat{p}$. Since $e$ avoids the pattern $102$ and $e_{p-1} \leq e_{p+1} < m$, 
it follows that 
$\hat{p} \geq i_1 = p$ so that $ b=b_1\leq 0$. 
\item If $k\geq 2$, let 
\[b_j=
\begin{cases}
  i_1 - i_2 + 2 & \mbox{for $j=1$,} \\
  i_j - i_{j+1} + 1 & \mbox{for $j=2, \dots, k-1$,}\\
  i_k - k - \hat{p} & \mbox{for $j=k$.} 
\end{cases}
\]
Since $\hat{p} \geq i_k - k$, we have $b_i\leq 0$ for all $j$, so $b\leq 0$.
\end{itemize}
In both cases, we have 
\[
  \prmx(\hat{e}) = \hat{p} = p - (b_1 + b_2 + \dots + b_k) = \prmx(e) - b.
\]
\end{enumerate}
Thus, the labeled $F$-path $\phi^{-1}(e)$ can be recovered recursively
by attaching the step $(a, b)$ with its label to the end of $\phi^{-1}(\hat{e})$.
Therefore, $\phi$ is bijective.
\end{proof}

%------------------------------------------------
\subsection{A bijection $\psi$ from $\mathcal{LF}_{n}$ to $\mathcal{UVD}_{n+1}$}
\label{sec:psi}

We construct a map from labeled $F$-paths to UVD paths.
Define 
$$\psi : \bigcup_{n\geq 0} \LFp{n}
\to \bigcup_{n\geq 0} \UVD_{n+1}$$
recursively, with the property that $\height(Q) = \vox(\psi(Q))$ as follows:

For the initial case $n=0$, we define $\psi((x_0,y_0)) := ud$, ensuring that 
$\height((x_0,y_0)) = \vox(ud) = 0$.
For the recursive step $n\geq 1$, consider a labeled $F$-path 
$Q = (x_0, y_0), (x_1, y_1), \dots, (x_\ell, y_\ell)$ in $\LFp{n}$ 
with $\ell \geq 1$. Let
$$\hat{Q}:=(x_0, y_0), (x_1, y_1), \dots, (x_{\ell-1}, y_{\ell-1})
\in \LFp{n-k},$$
where $k$ is the semilength of the last step of $Q$
with $0 < k \leq n$.
We define
\[
\hat{S} 
:= \hat{s}_1 \hat{s}_2 \dots \hat{s}_{2(n-k+1)+\hat{r}}
= \psi(\hat{Q}) \in \UVD_{n-k+1}
\]
where $\hat{r}$ is the number of vertical steps in $\hat{S}$.
Thus, $\height(\hat{Q}) = \vox(\hat{S})$, which we denote by $\hat{h}$.
Recall that $\hat{S}$ has exactly $\hat{h}+1$ returns.
Depending on the label of the last step $(a,b)$ of $Q$, we define $\psi(Q):=S$ differently.

\begin{enumerate}[(1)]
  \item If the label of the last step of $Q$ is $(a; 1)$,
  then $k=1$.
  Let $p$ be the index of the $(\hat{h}+1-a)$-th return in 
  $\hat{S}=\hat{s}_1 \hat{s}_2 \dots \hat{s}_{2n+\hat{r}}.$
  We decompose $\hat{S}$ as
  \begin{align*}
  \alpha &= \hat{s}_1 \hat{s}_2 \dots \hat{s}_{p} \quad \text{ and } \quad 
  \beta = \hat{s}_{p+1} \dots \hat{s}_{2n+\hat{r}},
  \end{align*} 
  so that $\vox(\alpha) = \hat{h}-a$ and $\vox(\beta) = a-1$,
  where $\alpha$ or $\beta$ may be empty.
 Define 
 $$S := \alpha \, u \, \beta \, d = \hat{s}_1 \hat{s}_2 \dots \hat{s}_{p} u \hat{s}_{p+1} \dots \hat{s}_{2n+\hat{r}} \, d.$$
 See Figure~\ref{fig:psicase1} for an illustration of $\hat{S}$ and $S$. 
 Since $\hat{s}_{p}=d$ and $\hat{s}_{p+1}=u$,
 the path $S$ contains neither $uv$ nor $vu$ and hence
 belongs to $\UVD_{n+1}$.
 Moreover, the subpath $u \beta d$ has no valleys on the $x$-axis, so
  \begin{align*}
    \vox(S) = \vox(\hat{S}) + (1-a) = \height(\hat{Q}) + (1-a) = \height(Q).
  \end{align*}

  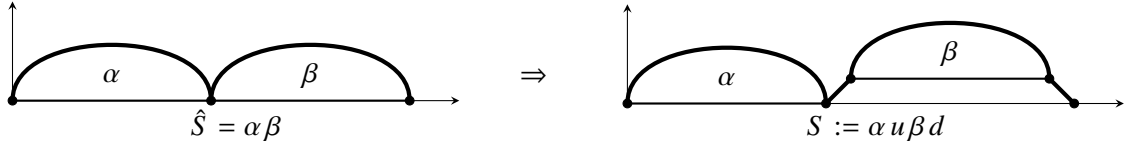
\begin{figure}
    %1----------------------------------------------------
    \begin{tikzpicture}[scale=0.33]
      \draw [-stealth] (0,0) -- (18,0);
      \draw [-stealth] (0,0) -- (0,4);

      \coordinate (0) at (0,0);
      \coordinate (1) at (8,0);
      \coordinate (2) at (16,0);

      \coordinate (a) at (0,3);
      \coordinate (b) at (8,3);
      \coordinate (c) at (16,3);

      \foreach \i in {0,...,2}{
      \filldraw (\i) circle (5pt);
      }

      \draw[ultra thick] (0) .. controls (a) and (b) .. (1);
      \draw[ultra thick] (1) .. controls (b) and (c) .. (2);
      \draw[thick] (0) -- (2);

      \node at (4, 1) {\( \alpha \)};
      \node at (12, 1) {\( \beta \)};

      \node at (9, -1) [] {\( \hat{S} = \alpha \, \beta \)};
      \node at (21,1) {$\Rightarrow$};
    \end{tikzpicture}
    \qquad 
  %2----------------------------------------------------
    \begin{tikzpicture}[scale=0.33]
      \draw [-stealth] (0,0) -- (20,0);
      \draw [-stealth] (0,0) -- (0,4);
  
      \coordinate (0) at (0,0);
      \coordinate (1) at (8,0);
      \coordinate (2) at (9,1);
      \coordinate (3) at (17,1);
      \coordinate (4) at (18,0);
  
      \coordinate (a) at (0,3);
      \coordinate (b) at (8,3);
      \coordinate (c) at (9,4);
      \coordinate (d) at (17,4);
  
      \foreach \i in {0,...,4}{
      \filldraw (\i) circle (5pt);
      }
  
      \draw[ultra thick] (0) .. controls (a) and (b) .. (1);
      \draw[ultra thick] (1) -- (2);
      \draw[ultra thick] (2) .. controls (c) and (d) .. (3);
      \draw[ultra thick] (3) -- (4);
      \draw[thick] (0) -- (1) -- (2) -- (3) -- (4);
  
      \node at (4, 1) {\( \alpha \)};
      \node at (13, 2) {\( \beta \)};
  
      \node at (10, -1) [] {\( S := \alpha\, u \, \beta \, d \)};
    \end{tikzpicture}
  \caption{
  The construction of $S$ from $\hat{S}$ in Case (1) of $\psi$. 
}
\label{fig:psicase1}
\end{figure}

  \item If the label of the last step $(a, b)$ of $Q$ is $(a; b_1, b_2, \dots, b_k)$
  for $a\geq 1$ and $b=b_1+b_2+\cdots+b_k \leq 0$, 
  then we define $j_i$ 
  by the index of $(\hat{h}+1+b_{i+1}+b_{i+2}+\dots+b_k-a)$-th return
  for $i=0, \dots, k$.
  (In detail, $j_k$ means the index of $(\hat{h}+1-a)$-th return of $\hat{S}$.)
  Let $p$ be the index satisfying that
  \[
  \hat{s}_p\neq v, \quad
  \hat{s}_{p+1} = \dots = \hat{s}_{j_0-1} = v,\quad
  \text{ and } \quad
  \hat{s}_{j_0} = d.
  \]
  Since
  \begin{align*}
  1 \leq p < j_0 \leq j_1 \leq \dots \leq j_k < 2(n-k+1)+\hat{r},
  \end{align*}
  the path $\hat{S}$ can be decomposed into 
  $k+3$
  paths based on the indices $p$, $j_0$, $\ldots$, $j_k$ as
  \begin{align*}
  \alpha &= \hat{s}_1 \hat{s}_2 \dots \hat{s}_{p}, \\
  \beta &= \hat{s}_{p+1} \dots \hat{s}_{j_0} (= v \dots v d), \\
  \sigma_i &= \hat{s}_{j_{i-1}+1} \dots \hat{s}_{j_i} \quad \text{for $i=1, \dots, k$},\\
  \tau &= \hat{s}_{j_k+1} \dots \hat{s}_{2(n-k+1)+\hat{r}},
  \end{align*} 
  with $\vox(\alpha\beta)= \hat{h} + (b - a)$, $\vox(\sigma_j) = -b_j-1$ for $j=1,\dots,k$ and $\vox(\tau) = a-1$.
  Note that $\alpha$, $\beta$, and $\tau$ are nonempty and $\sigma_i$ is empty if $j_{i-1} = j_{i}$.
  From $\hat{S}$, we construct a UVD path $S$ 
  by inserting $u\sigma_i u$'s in order after $\alpha$, 
  followed by $\tau$, $k$ instances of $v$, and $\beta$:
  \begin{align*}
  S := 
  & \alpha \; u \sigma_1 u\; u \sigma_2 u \dots u \sigma_k u \; \tau \; v^k \; \beta \\
  =&\hat{s}_1 \hat{s}_2 \dots \hat{s}_{p} \,
  u \hat{s}_{j_0+1} \dots \hat{s}_{j_1} u\, 
  u \hat{s}_{j_1+1} \dots \hat{s}_{j_2} u\, \dots \\
  &\dots 
  u \hat{s}_{j_{k-1}+1} \dots \hat{s}_{j_k} u\, 
  \hat{s}_{j_k+1} \dots \hat{s}_{2(n-k+1)+\hat{r}} \, 
  v \dots v \, 
  \hat{s}_{p+1} \dots \hat{s}_{j_0}.
  \end{align*}
  See Figure~\ref{fig:psicase2} for $\hat{S}$ and $S$. 
  It is easy to check that 
  the last lattice point of $S$ is $(2n+2, 0)$ and
  $S$ has $(\hat{r}+k)$ vertical steps. 
  Letting $r:=\hat{r} + k$,
  the number of all steps of $S$ is $2(n+1)+r$.
  It is obvious that $S$ has neither $uv$ nor $vu$,
  and $S$ is in $\mathcal{UVD}_{n+1}$
  with 
  \begin{align*}
  \vox(S) &= \vox(\hat{S}) + (b_1+b_2+\dots + b_k -a)
  = \height(\hat{Q}) + (b -a) = \height(Q).
  \end{align*}
\end{enumerate}
Thus, $\psi(Q)$ can be defined by $S$ in $\UVD_{n+1}$ in both cases.

 \begin{figure}  
   %1----------------------------------------------------
    \begin{tikzpicture}[scale=0.33]
      \draw [-stealth] (0,0) -- (30,0);
      \draw [-stealth] (0,0) -- (0,7);

      \coordinate (0) at (0,0);
      \coordinate (1) at (7,5);
      \coordinate (2) at (7,3);
      \coordinate (3) at (7,1);
      \coordinate (4) at (8,0);
      \coordinate (5) at (12,0);
      \coordinate (6) at (16,0);
      \coordinate (7) at (18,0);
      \coordinate (8) at (22,0);
      \coordinate (9) at (28,0);

      \coordinate (a) at (0,6);
      \coordinate (b) at (4,8);
      \coordinate (c) at (8,3);
      \coordinate (d) at (12,3);
      \coordinate (e) at (16,3);
      \coordinate (f) at (18,3);
      \coordinate (g) at (22,3);
      \coordinate (h) at (22,4);
      \coordinate (i) at (28,4);

      \draw[ultra thick] (0) .. controls (a) and (b) .. (1);
      \draw[ultra thick, color=red] (1) -- (2) -- (3) -- (4);
      \draw[ultra thick] (4) .. controls (c) and (d) .. (5);
      \draw[ultra thick] (5) .. controls (d) and (e) .. (6);
      \draw[ultra thick] (7) .. controls (f) and (g) .. (8);
      \draw[ultra thick] (8) .. controls (h) and (i) .. (9);
      \draw[thick] (0) -- (4) -- (5) -- (6) -- (7) -- (8) -- (9);

      \foreach \i in {0,...,9}{
      \filldraw (\i) circle (5pt);
      }

      \node at (4, 1) {\( \alpha \beta \)};
      \node at (10, 1) {\( \sigma_1 \)};
      \node at (14, 1) {\( \sigma_2 \)};
      \node at (17, 1) {\( \dots \)};
      \node at (20, 1) {\( \sigma_k \)};
      \node at (25, 1) {\( \tau \)};

      \node at (15, -1) [] 
      {\( 
      \hat{S} := \alpha \, \beta \, \sigma_1 \, \sigma_2 \dots \sigma_k \tau
      \)};
      \node at (15, -4) {$\Downarrow$};
    \end{tikzpicture}\\
  
   %2----------------------------------------------------
    \begin{tikzpicture}[scale=0.38]
      \draw [-stealth] (0,0) -- (36,0);
      \draw [-stealth] (0,0) -- (0,17);

      \coordinate (0) at (0,0);
      \coordinate (1) at (7,5);
      \coordinate (2) at (8,6);
      \coordinate (3) at (12,6);
      \coordinate (4) at (13,7);
      \coordinate (5) at (14,8);
      \coordinate (6) at (18,8);
      \coordinate (7) at (19,9);
      \coordinate (8) at (21,11);
      \coordinate (9) at (22,12);
      \coordinate (10) at (26,12);
      \coordinate (11) at (27,13);
      \coordinate (12) at (33,13);
      \coordinate (13) at (33,11);
      \coordinate (14) at (33,9);
      \coordinate (15) at (33,7);
      \coordinate (16) at (33,5);
      \coordinate (17) at (33,3);
      \coordinate (18) at (33,1);
      \coordinate (19) at (34,0);

      \coordinate (a) at (0,6);
      \coordinate (b) at (4,8);
      \coordinate (c) at (8,9);
      \coordinate (d) at (12,9);
      \coordinate (e) at (14,11);
      \coordinate (f) at (18,11);
      \coordinate (g) at (22,15);
      \coordinate (h) at (26,15);
      \coordinate (i) at (27,17);
      \coordinate (j) at (33,17);

      \coordinate (k) at (7,0);

      \draw[ultra thick] 
        (0) .. controls (a) and (b) .. (1)
        (1) -- (2) .. controls (c) and (d) .. (3) -- (4)
        (4) -- (5) .. controls (e) and (f) .. (6) -- (7)
        (8) -- (9) .. controls (g) and (h) .. (10) -- (11)
        (11) .. controls (i) and (j) .. (12)
        (12) -- (13) 
        (14) -- (15) -- (16) 
      ;
      \draw[ultra thick, color=red]  (16)-- (17) -- (18) -- (19);
      
      \foreach \i in {0,...,19}{
      \filldraw (\i) circle (5pt);
      }
      
      \draw[thick] 
        (0) -- (k) 
        (1) -- (2) -- (3) -- (5) -- (6)
        (9) -- (10) -- (11) -- (12)
      ;
      \draw[dotted, thick] 
        (7) -- (8)
        (13) -- (14)
      ;
      \draw[dashed] 
        (k) -- (1)
        (1) -- (16)
        (4) -- (15)
        (7) -- (14)
        (8) -- (13)
      ;

      \node at (4, 2) {\( \alpha \)};
      \node at (10, 7) {\( \sigma_1 \)};
      \node at (16, 9) {\( \sigma_2 \)};
      \node at (24, 13) {\( \sigma_k \)};
      \node at (30, 14) {\( \tau \)};
      \node at (18) [above left] {\( \beta \)};

      \foreach \i in {2,4,5,7,9,11}{
        \node at (\i) [below] {\( u \)};
      }
      
      \foreach \i in {13, 15, 16}{
        \node at (\i) [above left] {\( v \)};
      }

      \node at (18, -1) [] 
      {\( S := \alpha \; u \sigma_1 u\; u \sigma_2 u \dots u \sigma_k u \; \tau \; v^k \; \beta \)};
    \end{tikzpicture}
\caption{
  The construction of $S$ from $\hat{S}$ in Case (2) of $\psi$. The red steps indicate the steps in $\beta$.
}
\label{fig:psicase2}
\end{figure}
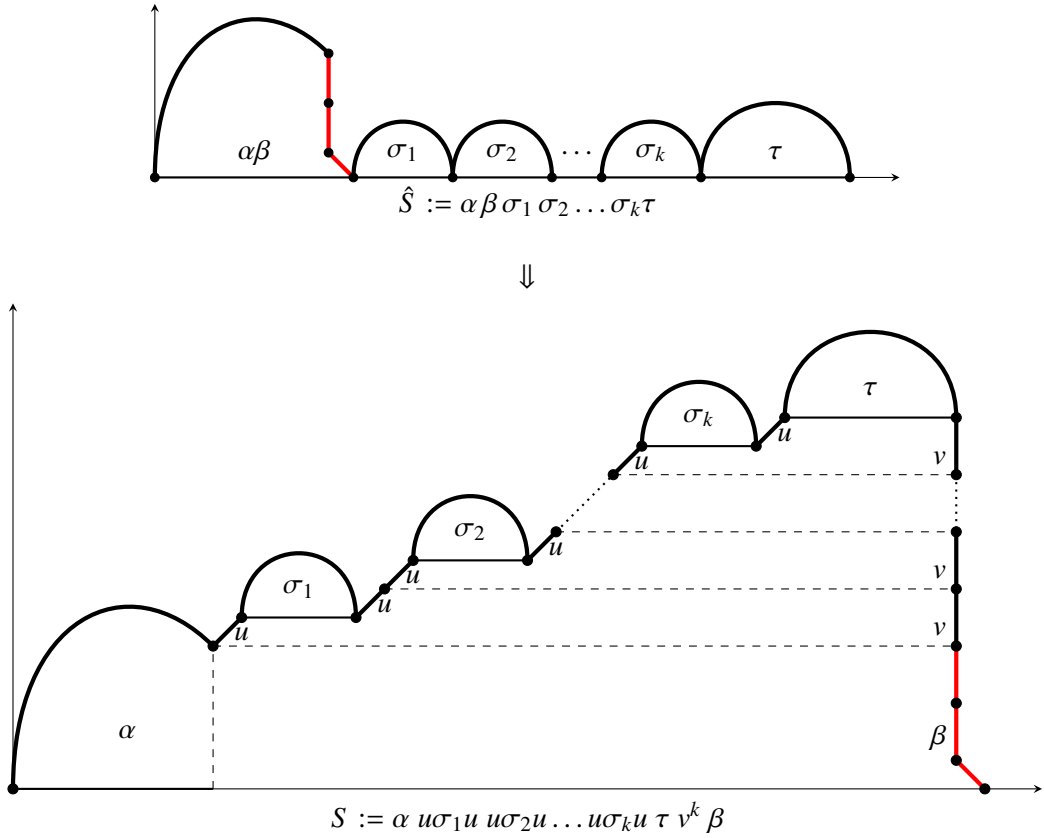  

Continuing with the 
labeled $F$-path $Q$ from Figure~\ref{fig:F} 
and the 
labeled $F$-paths $Q^{(0)}$, $\ldots$, $Q^{(19)}$ introduced in Subsection~\ref{sec:phi},
we now apply the map $\psi$ recursively
to obtain the corresponding UVD path $\psi(Q)$.
% $$\psi(Q) = uduududuuddduduuuududuuduuduuuuduuuuuduuuuuuuuudvvvdvvvvvvd,$$
%----------------------------------------------------
\begin{figure}[h]

\[
  \begin{tikzpicture}[scale=0.27]
    \pgfmathsetmacro{\x}{51}
    \pgfmathsetmacro{\y}{22}
    \pgfmathsetmacro{\xx}{\x+1}
    \pgfmathsetmacro{\yy}{\y+1}
    \draw[gray!40, ultra thin] (0,0) grid (\x,\y);
    \draw [-stealth] (0,0) -- (\xx,0);
    \draw [-stealth] (0,0) -- (0,\yy);
    
    \coordinate (0) at (0,0);
    \coordinate (1) at (1,1);
    \coordinate (2) at (2,0);
    \coordinate (3) at (3,1);
    \coordinate (4) at (4,2);
    \coordinate (5) at (5,1);
    \coordinate (6) at (6,2);
    \coordinate (7) at (7,1);
    \coordinate (8) at (8,2);
    \coordinate (9) at (9,3);
    \coordinate (10) at (10,2);
    \coordinate (11) at (11,1);
    \coordinate (12) at (12,0);
    \coordinate (13) at (13,1);
    \coordinate (14) at (14,0);
    \coordinate (15) at (15,1);
    \coordinate (16) at (16,2);
    \coordinate (17) at (17,3);
    \coordinate (18) at (18,4);
    \coordinate (19) at (19,3);
    \coordinate (20) at (20,4);
    \coordinate (21) at (21,3);
    \coordinate (22) at (21,3);
    \coordinate (23) at (22,4);
    \coordinate (24) at (23,5);
    \coordinate (25) at (24,4);
    \coordinate (26) at (25,5);
    \coordinate (27) at (26,6);
    \coordinate (28) at (27,5);
    \coordinate (29) at (28,6);
    \coordinate (30) at (29,7);
    \coordinate (31) at (30,8);
    \coordinate (32) at (31,9);
    \coordinate (33) at (32,8);
    \coordinate (34) at (33,9);
    \coordinate (35) at (34,10);
    \coordinate (36) at (35,11);
    \coordinate (37) at (36,12);
    \coordinate (38) at (37,13);
    \coordinate (39) at (38,12);
    \coordinate (40) at (39,13);
    \coordinate (41) at (40,14);
    \coordinate (42) at (41,15);
    \coordinate (43) at (42,16);
    \coordinate (44) at (43,17);
    \coordinate (45) at (44,18);
    \coordinate (46) at (45,19);
    \coordinate (47) at (46,20);
    \coordinate (48) at (47,21);
    \coordinate (49) at (48,20);
    \coordinate (50) at (48,18);
    \coordinate (51) at (48,16);
    \coordinate (52) at (48,14);
    \coordinate (53) at (49,13);
    \coordinate (54) at (49,11);
    \coordinate (55) at (49,9);
    \coordinate (56) at (49,7);
    \coordinate (57) at (49,5);
    \coordinate (58) at (49,3);
    \coordinate (59) at (49,1);
    \coordinate (60) at (50,0);

    \foreach \i in {0,...,60}{
      \filldraw (\i) circle (5pt);
    } 

    \foreach \i in {1,...,60}{
      \pgfmathsetmacro{\j}{\i-1}
      \draw[ultra thick] (\j) -- (\i);
    }

    % \node at (0, \y) [left] {$\psi(Q)$};
    \node at (0, 0) [below left] {$O$};
    \foreach \i in {5,10,...,50}{
      \node at (\i, 0) [below] {$\i$};
    }
    \foreach \i in {5,10,...,22}{
      \node at (0,\i) [left] {$\i$};
    }
  \end{tikzpicture}
\]
\caption{The UVD path $\psi(Q)$ corresponding to $Q$ in Figure~\ref{fig:F}.}
\label{fig:UVD}
\end{figure}
%----------------------------------------------------
Defining $S^{(j)} = \psi(Q^{(j)})$, 
in the process of finding $\psi(Q)$, 
we obtain some $S^{(j)}$'s as follows.
\begin{align*}
% S^{(0)}  &= ud, \\
S^{(3)}  &= ud\;ud\;ud\;ud, \\
S^{(4)}  &= ud\;ud\;ud\;uudd \\
S^{(5)}  &= ud\;uududuuddd, \\
S^{(9)}  &= ud\;uududuuddd\;ud\;ud\;ud\;ud, \\
S^{(10)} &= ud\;uududuuddd\;ud\;uuuududvd, \\
S^{(12)} &= ud\;uududuuddd\;ud\;uuuududvd\;ud\;ud, \\
S^{(13)} &= ud\;uududuuddd\;ud\;uuuududuuduudvvd, \\
S^{(17)} &= ud\;uududuuddd\;ud\;uuuududuuduudvvd\;ud\;ud\;ud\;ud, \\
S^{(18)} &= ud\;uududuuddd\;ud\;uuuududuuduudvvd\;ud\;ud\;uuuuuuuudvvvd, \\
S^{(19)} &= ud\;uududuuddd\;ud\;uuuududuuduuduuuuduuuuuduuuuuuuuudvvvdvvvvvvd.
\end{align*}
Here, a valley $du$ on the $x$-axis is indicated by a small gap between $d$ and $u$
so that valleys can be visually distinguished.
See Figure~\ref{fig:UVD} for the UVD path $\psi(Q)$.

\begin{thm}
\label{thm:UVD}
The map 
$$\psi : \bigcup_{n\geq 0} \LFp{n}
\to \bigcup_{n\geq 0} \UVD_{n+1}$$
is a bijection with $\height(Q) = \vox(\psi(Q))$.
\end{thm}

\begin{proof}
  As with $2$-Schr\"{o}der paths, 
  we denote the number of returns of a UVD path $S$ by $\block(S)$.
  For a given $Q=(x_0, y_0), (x_1, y_1), \dots, (x_{\ell}, y_{\ell})$ in $\mathcal{LF}_{n}$, 
  one can easily verify that
  $$ \block(\psi(Q)) = \vox(\psi(Q)) + 1 = \height(Q) +1.$$
  
  To prove the bijectivity of $\psi$, it suffices to show that, given a UVD 
  path $S$, one can uniquely reconstruct a labeled $F$-path $Q$ such that $S=\psi(Q)$,
  by determining the 
  labeled $F$-path $\hat{S}=\psi(\hat{Q})$ 
  and the last step
  $(a,b)$ of $Q$, along with its label. \\
  
  For $n\geq 1$, let $S$ be a UVD path in $\UVD_{n+1}$. 
  There are two cases to consider depending on the second-to-last step of $S$.
  
  \begin{enumerate}[(1)]
    \item 
    If the second-to-last step of $S$ is not vertical, 
    then $S$ can be decomposed as $\alpha \, u \, \beta \, d$,
    where both $\alpha$ and $\beta$ are (possibly empty) UVD paths.
    In this case, we recover
    $\hat{S} = \alpha \beta$, which is in $\UVD_{n}$ with 
    \begin{align*}
      \block(\hat{S}) 
      &= \block(\alpha) + \block(\beta) 
      = \block(S) - 1 + \block(\beta).
    \end{align*}
    Thus, we recover $(a,b)=(a,1)$ with label $(a; 1)$, where $a=\block(\beta)$.
    
    \item If the second-to-last step of $S$ is vertical,
    then there exists a unique positive integer $k$
    such that
    $S$ can be expressed as 
    \begin{align*}
    S = \alpha \; u \sigma_1 u\; u \sigma_2 u \dots u \sigma_k u \; \tau \; v^k \; \beta, 
    \end{align*}
    where 
    $\alpha \beta$, $\sigma_1$, $\dots$, $\sigma_k$, and $\tau$ 
    are UVD paths,
    $\beta$ is of the form $v \dots v d$, 
    and $\sigma_i$ may be empty.
    In this case, we recover 
    $$
    \hat{S} = \alpha \, \beta \, \sigma_1 \, \sigma_2 \dots \sigma_k \tau,
    $$
    which is in $\mathcal{UVD}_{n-k+1}$ with 
    \begin{align*}
      \block(\hat{S}) 
      &= \block(\alpha\beta) + \block(\sigma_1) + \dots + \block(\sigma_k) 
      + \block(\tau)\\
      &= \block(S) -(b_1 + \dots +b_k - a),
    \end{align*}
    where 
    $a = \block(\tau) \geq 1$ and
    $b_j = -\block(\sigma_j) \leq 0$ for $j=1, \dots, k$.
    Thus, we recover $(a,b)$ with label $(a; b_1, \dots, b_k)$.
  \end{enumerate}
  Therefore, the labeled $F$-path $\psi^{-1}(S)$ can be recovered recursively
    by attaching the step $(a,b)$ with its label
    to the end of $\psi^{-1}(\hat{S})$. Thus, $\psi$ is bijective.
\end{proof}

Combining the three maps $M$, $\phi$, and $\psi$, for a positive integer $n$, the map 
$$\phi \circ \psi^{-1} \circ M:\SP_{n} \to \IS_{n}(102)$$ 
is a bijection 
from the set of $2$-Schr\"{o}der paths having neither peaks $NE$ nor valleys $EN$ 
ending with a diagonal step $D$ 
to the set of $102$-avoiding inversion sequences.
In particular, $\block (P) = \rank (\phi \circ \psi^{-1} (P))$ for a path $P \in \SP_{n}$. 
The following diagram visually represents the connections and mappings between the sets
$\SP_{n}$, $\LFp{n-1}$, $\UVD_{n}$, and $\IS_{n}(102)$, along with the labels for each map.

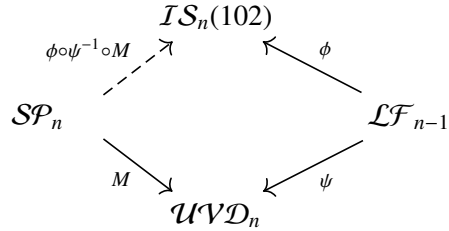
\[
\begin{tikzcd} 
  & 
  \IS_{n}(102) \\
  \mathcal{SP}_n
    \arrow[rd, start anchor={[xshift=2ex]}, "M"' ] 
    \arrow[ru, start anchor={[xshift=2ex]}, dashed, "\phi \circ \psi^{-1} \circ M"] &&
  \mathcal{LF}_{n-1} 
    \arrow[lu, "\phi"'] 
    \arrow[ld, "\psi"] \\
  & 
  \mathcal{UVD}_{n} 
\end{tikzcd}
\]

%===============================================

\section{Refined enumeration of $102$-avoiding inversion sequences}
\label{sec:102-with-rank}

In this section, we provide the number of $102$-avoiding inversion sequences with a fixed rank using
the bijection $\phi \psi^{-1}$ and Lagrange inversion formula.
For nonnegative integers $n$ and $t$, let 
$\IS_{n,t}(102)$ denote the set of $e \in \IS_n(102)$ with $\rank (e) = t$ and
\(\mathcal{UVD}_{n,t} \) denote the set of $S\in \mathcal{UVD}_{n}$ with $\vox(S)=t$. 
Let
\begin{gather*}
D=D(x):=\sum_{n\ge 1 } | \mathcal{UVD}_{n} | \,x^n,\qquad 
D_t=D_t(x):=\sum_{n\ge 1 } | \mathcal{UVD}_{n,t} | \,x^n.
\end{gather*} 
By decomposition of UVD paths in $\cup_{n\ge1}\mathcal{UVD}_{n}$, we have
\begin{gather*}
D_t =D_0^{t+1},\qquad
D=\sum_{t\ge0}D_t=\frac{D_0}{1-D_0}.
\end{gather*} 

Recall that for a UVD path $S$, $S$ has the semilength $n$ if and only if the sum of the numbers of $d$ and $v$ in $S$ equals to $n$.  
Given a UVD path $S$, $S$ ends with either $v^0 d$ or $v^k d$ for some positive integer $k$.
So we get 
\begin{align*}
D&=x(1+D)^2 +\sum_{k\geq1}x^{k+1}(1+D)^{2k+1}D\\
&=x(1+D)^2 + \frac{x^2(1+D)^3D}{1-x(1+D)^2}.
\end{align*}
Multiplying both sides by $1-x(1+D)^2$, we obtain
$$
D=(x-x^2)(1+D)^3.
$$
Since $D=\frac{D_0}{1-D_0}$, we get
$$
D_0=(x-x^2)\frac{1}{(1-D_0)^2}
$$
and letting $y=x-x^2$, we have
\begin{equation}\label{eqn:formE}
E(y)= y\frac{1}{(1-E(y))^2},
\end{equation}
where the power series 
$E(y)$ satisfies $E(x-x^2)=D_0(x)$.
By applying
Lagrange inversion formula \cite[eq. (2.2.1)]{Ges16} to \eqref{eqn:formE}, we obtain
\begin{equation*}
\left[y^n\right]E(y)^{t+1}
=\frac{t+1}{n}\left[z^{n-t-1}\right]\left( \frac{1}{(1-z)^2}\right)^n 
= \frac{t+1}{n}\binom{3n-t-2}{n-t-1}.
\end{equation*}
Thus we have 
\begin{equation}\label{eqn:sumD}
D_0(x)^{t+1}=\sum_{n\ge t+1} \frac{t+1}{n}\binom{3n-t-2}{n-t-1} \left(x-x^2\right)^n.
\end{equation}
Since $D_t=D_0^{t+1}$, we obtain the following theorem
from Theorems~\ref{thm:IS} and \ref{thm:UVD} and \eqref{eqn:sumD}.
\begin{thm}
Let $n$ and $t$ be integers such that $n \ge 1$ and $0 \le t \le n-1$. Then we have
$$| \IS_{n,t}(102)| = | \mathcal{UVD}_{n,t} | = \sum_{j =t+1}^{n} (-1)^{n-j}\frac{t+1}{j}\binom{3j-t-2}{j-t-1}\binom{j}{n-j}.$$
\end{thm}

%===================================================

\section{\( (102, \tau) \)-avoiding inversion sequences with a fixed rank}
\label{sec:102-tau-with-rank}

The enumeration of inversion sequences avoiding the pattern \(102\) 
and another pattern \(\tau\) of length 3
has been previously studied by 
Corteel et al.~\cite{CMSW16}, 
Huh et al.~\cite{HKSS24},
Kotsireas et al.~\cite{KMY24},
Testart~\cite{Testart24}, and
Yan and Lin~\cite{YanLin20}. 
Among these, the \((102,101)\)-avoiding case has been particularly well-studied in~\cite{HKSS24},
where several statistics on \(\IS_n(102,101)\) were investigated and explicitly enumerated.
In particular, our rank statistic appears naturally in that context as one of the statistics considered.
One of their results can be restated as follows (compare with \(a_n(*,*,m)\) in Corollary~6.3 of~\cite{HKSS24}). 

\begin{prop}[\cite{HKSS24}] 
For integers \(n \geq 2\) and \(0 \leq t \leq n-2\), the number of 
\((102,101)\)-avoiding inversion sequences \(e\) of length \( n \) with \(\rank(e) = t\) is given by
\[
\frac{t+1}{n} \sum_{i=1}^{n-t-1} \binom{n}{i} \binom{n-t+i-2}{2i-1}.
\]
\end{prop}

Motivated by this, we extend the study to the enumeration of inversion sequences of length \( n \)
that avoid both patterns \( 102 \) and \( \tau \in \{001, 011, 012, 021, 110, 120, 201, 210\} \), 
and have fixed rank \( t \). 
Here, we exclude the patterns \(000\), \(010\), and \(100\) from the set of candidates for \(\tau\),
as these cases do not easily admit a generalizable description in terms of the rank statistic.
 
We denote the set of such sequences by 
\[
\IS_{n,t}(102, \tau) := \{ e \in \IS_n(102, \tau) : \rank(e) = t \}. 
\]
Note that \(\IS_{n,n-1}(102, \tau)\) is the singleton set \(\{(0, 0, \dots, 0)\}\) or the empty set, 
depending on the choice of \(\tau\) and \( n \).
Therefore, we restrict our attention to the range \(n \geq 2\) and \(0 \leq t \leq n - 2\).

%-----------------------------------------------------------------------------
\subsection{(102, 001)-avoiding inversion sequences}

It was shown in \cite{CMSW16} that \( e \in \IS_{n}(001) \) if and only if \( e \) is of the form
\begin{equation}\label{eq:smsw}
e_1<e_2<\cdots<e_{k}\geq e_{k+1}\geq \cdots \geq e_n
\end{equation}
for some \( k \), and furthermore, that every 
\( 001 \)-avoiding inversion sequence also avoids the pattern \( 102 \). Therefore, we conclude that 
\( e \in \IS_{n}(102,001) \) if and only if \( e \) is of the form \eqref{eq:smsw}.

\begin{prop}\label{prop:102001}
For integers \(n \geq 2\) and \(0 \leq t \leq n-2\), the number of 
\((102,001)\)-avoiding inversion sequences \(e\) of length \( n \) with \(\rank(e) = t\) is given by
\( 2^{n - t - 2} \).
\end{prop}

\begin{proof}
Consider inversion sequences \( e=(e_1,e_2, \dots, e_n) \). 
From \eqref{eq:smsw}, we see that \( e \in \IS_{n,t}(102,001) \) if and only if \( e \) is of the form
\[
e_i=i-1 \mbox{ for \( 1\leq i \leq m \)}, \qquad 
m=e_{m+1}=\cdots =e_{m+t+1} > e_{m+t+2} \geq \cdots \geq e_n
\]
for some \( m \).  The number of such sequences is given by 
\[ 
\sum_{m=1}^{n-t-1}\binom{n-t-2}{m-1} =2^{n-t-2}.
\] 
\end{proof}

%-----------------------------------------------------------------------------

\subsection{(102, 011)-avoiding inversion sequences}

We begin with the following lemma, which provides a finer understanding of the structure of \((102,011)\)-avoiding inversion sequences.

\begin{lem}\label{lem:102011}
For positive integers \( n \) and \( t \) with \( t \leq n-1 \), we have
\[
| \IS_{n+1,t}(102,011) |=| \IS_{n,t-1}(102,011)|.
\]
\end{lem}

\begin{proof}
We prove the lemma by exhibiting a bijection.

Let \( e = (e_1, e_2, \dots, e_{n+1}) \in \IS_{n+1,t}(102,011) \) with \( t > 0 \).  
Define a map
\( \varphi(e) := (e_2, \dots, e_{n+1})\).
To show that \( \varphi(e) \in \IS_{n,t-1}(102,011) \), observe that  
if \( e_i = i - 1 \) for some \( i \geq 2 \), then by the structure of \( (102,011) \)-avoiding sequences,  
we must have \( e_j = j - 1 \) for all \( j \geq i \) up to the maximum value,  
which forces \( \rank(e) = 0 \), contradicting the assumption \( t > 0 \).  
Hence, \( e_i < i - 1 \) for all \( i \geq 2 \), and the tail sequence lies in \( \IS_{n}(102,011) \) with rank \( t - 1 \).
The map \( \varphi \) is clearly bijective, completing the proof.
\end{proof}

Yan and Lin~\cite[Theorem~3.1]{YanLin20} showed that 
the total number of \((102,011)\)-avoiding inversion sequences of length \(n\) is given by
\begin{equation}\label{eq:fibo}
|\IS_{n}(102,011)| = F_{2n-1},
\end{equation}
where \( \{F_n\} \) is the Fibonacci sequence, 
defined by the recurrence \( F_n = F_{n-1} + F_{n-2} \) with initial values 
\( F_0 = 0 \) and \( F_1 = 1 \).

As a direct consequence of Lemma~\ref{lem:102011}, we now refine this result by enumerating the number of such sequences according to their rank.

\begin{prop}\label{prop:102011}
For integers \(n \geq 2\) and \(0 \leq t \leq n-2\), the number of 
\((102,011)\)-avoiding inversion sequences \(e\) of length \( n \) with \(\rank(e) = t\) is given by
\( F_{2n-2t-2} \), where \( F_k \) denotes the \( k \)-th Fibonacci number.
\end{prop}

\begin{proof}
We proceed by induction on \( n \geq 2 \).
For the base case \( n = 2 \) and \( t = 0 \), we have
\[
|\IS_{2,0}(102,011)| = |\{01\}| = 1 = F_2.
\]

Now assume that the proposition holds for some \( n \geq 2 \) and all \( 0 \leq t \leq n - 2 \).
By Lemma~\ref{lem:102011} and the induction hypothesis, for \( 1 \leq t \leq n - 1 \), we have
\[
|\IS_{n+1,t}(102,011)| = |\IS_{n,t-1}(102,011)| = F_{2n - 2(t - 1) - 2} = F_{2(n+1) - 2t - 2}.
\]
Moreover, by~\eqref{eq:fibo}, it follows that
\[
\sum_{t = 1}^{n} |\IS_{n+1,t}(102,011)| = \sum_{t = 1}^{n} |\IS_{n,t-1}(102,011)| = F_{2n - 1},
\]
and
\[
|\IS_{n+1,0}(102,011)| = F_{2n+1} - F_{2n-1} = F_{2n} = F_{2(n+1) - 2}.
\]
Therefore, the proposition holds for \( n+1 \), completing the induction step.
\end{proof}

%-----------------------------------------------------------------------------

\subsection{(102, 012)-avoiding inversion sequences}

The Fibonacci numbers \( F_{n+1} \) are well-known to count the number of tilings of a \( 1 \times n \) board using unit squares \( S \) and dominoes \( D \), where each domino covers two adjacent cells.  

In light of the identity
\begin{equation}\label{eq:fibo2}
|\IS_n(102,012)| = F_{2n - 1},
\end{equation}
established by Yan and Lin~\cite[Theorem~3.1]{YanLin20},  
we construct a bijection between \((102,012)\)-avoiding inversion sequences of length \(n\) 
and domino tilings of a board of length \(2n - 2\).  
This bijection allows a natural refinement by the rank statistic as follows.

\begin{prop}\label{prop:102012}
For integers \(n \geq 2\) and \(0 \leq t \leq n-2\), the number of 
\((102,012)\)-avoiding inversion sequences \(e\) of length \( n \) with \(\rank(e) = t\) is given by
\( (t+1)F_{2n-2t-3} \), where \( F_k \) denotes the \( k \)-th Fibonacci number.
\end{prop}

\begin{proof}
It is straightforward to verify that \( e= (e_1, e_2, \dots, e_n)  \in \IS_n(102,012) \) 
if and only if \( e \) is of the form
\[
m = e_{i_1} \geq e_{i_2} \geq \cdots \geq e_{i_\ell} \geq 1 \quad \text{for some } m < i_1 \leq i_2 \leq \cdots \leq i_\ell,
\]
with \( e_j = 0 \) for all other indices. Here, \( m = \max(e) \).

We define a map \( \varphi: \IS_n(102,012) \to \mathcal{T}_{2n-2} \), where \( \mathcal{T}_{2n-2} \) denotes the set of tilings of a \(1 \times (2n - 2)\) board with squares and dominoes. 

We set \( \varphi((0, 0, \dots, 0)) := D^{n-1} \), which corresponds to a tiling composed entirely of dominoes.

For an inversion sequence \( e \) with \( \max(e) = m > 0 \), let \( e_{i_1}, e_{i_2}, \dots, e_{i_\ell} \) be the nonzero entries of \( e \), forming a weakly decreasing sequence. For convenience, we set 
\( b_j := e_{i_j} - e_{i_{j+1}} \) for \( 1 \leq j \leq \ell-1 \), and \( b_{\ell}:=e_{i_\ell}-1 \).

Now define 
\[
\varphi(e) := \xi(e_{m+1})\, \xi(e_{m+2})\, \dots\, \xi(e_n),
\]
where
\[
\xi(e_k) :=
\begin{cases}
S D^{b_j} S & \mbox{ \( k=i_j \) for some \( j \)}, \\
D & \mbox{otherwise}.
\end{cases}
\]
For example, when \( n = 3 \), the following correspondences hold:
\[
000 \mapsto DD, \quad
001 \mapsto DSS, \quad
011 \mapsto SSSS, \quad
010 \mapsto SSD, \quad
002 \mapsto SDS.
\]
In general, the resulting tiling \( \varphi(e) \) consists of \( 2\ell \) squares and \( n - \ell - 1 \) dominoes, 
totaling \( 2n - 2 \) units in length, as desired. Hence, it is easy to recover \( \ell, m \), and 
the original sequence from the tiling, confirming that the map is invertible.

Next, we examine how the rank of \( e \) corresponds to the structure of the tiling \( \varphi(e) \). 

If \( \rank(e) = t \) with \( 0 \leq t < n - 2 \), then we must have \( m = \max(e) > 1 \), and the entries 
in positions from \( m+1 \) to \( m + t + 1 ~(= \prmx(e)) \) take the form
\[
(e_{m+1}, e_{m+2}, \dots, e_{m+t+1}) = (0^i, m^{t - i + 1}),
\]
for some \( 0 \leq i \leq t \), with the condition that \( e_{m + t + 2} < m \).
Under the bijection \( \varphi \), this corresponds to a tiling whose initial segment is of the form
\[
D^i S^{2t - 2i + 2} D \quad \text{or} \quad D^i S^{2t - 2i + 1} D.
\]
The former case arises when \( e_{m+t+2} = 0 \) and there is another occurrence of \( m \) 
at some later position \( k > m + t + 2 \). 
These two types of initial segments account for all possible configurations of inversion sequences with 
rank \( t<n-2 \), 
and their contributions sum to
\[
(t+1)(F_{2n - 2t - 5} + F_{2n - 2t - 4}) = (t+1)F_{2n - 2t - 3},
\]
as desired. Here, the factor \( t+1 \) corresponds to the number of possible choices for \( i \), and 
\( F_{2n - 2t - 5}  \) (respectively \( F_{2n - 2t - 4} \)) counts the number of tilings of 
a \( 1 \times (2n-2t-6) \) board (respectively \( 1 \times (2n-2t-5) \) board).

For the case \( t = n - 2 \), the inversion sequences are precisely those of the form \( (0^{n-i}, 1^{i}) \) for some \( 1 \leq i \leq n - 1 \). 
Thus, they correspond to tilings of the form 
\( D^{n - i - 1} S^{2i} \),
and the total number of such tilings is \( n - 1 \). This agrees with the value of 
\( (t+1) F_{2n - 2t - 3} \)
when \( t = n - 2 \), since \( F_1 = 1 \). The proof is complete.
\end{proof}

%-----------------------------------------------------------------------------

\subsection{(102, 021)-avoiding inversion sequences}

A \emph{Dyck path of semilength \( n \)} is a UVD path of semilength \( n \) 
that consists only of up steps \( u \) and down steps \( d \), with no vertical steps.
The following is a well-known result on Dyck paths. 

\begin{lem}\label{lem:Dyck_end}
Let \( k \) and \( n \) be positive integers with \( k\leq n \). The number of Dyck paths of semilength \( n \) 
ending with the subpath \( ud^k \) is equal to the number of Dyck paths \( P \) of semilength \( n \) 
with \( k \) returns. The number is given by
\[
\frac{k}{n}\binom{2n-k-1}{n-1}.
\]
\end{lem}

The sequence in Lemma~\ref{lem:Dyck_end} appears in OEIS~\cite{OEIS} as entry A33184. 
We now apply Lemma~\ref{lem:Dyck_end} to enumerate 
\( (102,021) \)-avoiding inversion sequences according to their rank.

\begin{prop}\label{prop:102021}
For integers \(n \geq 2\) and \(0 \leq t \leq n-2\), the number of 
\((102,021)\)-avoiding inversion sequences \(e\) of length \( n \) with \(\rank(e) = t\) is given by
\[
(t+1)\left( 2^{n-t-2} -(n-t-1) +\sum_{m=1}^{n-t-1} \frac{1}{m+t+1}\binom{2m+t}{m} \right).
\]
\end{prop}

\begin{proof}
We begin by considering inversion sequences \( e=(e_1, e_2, \dots, e_n) \) with \( \max(e)=m>0 \). 

It is straightforward to verify that \( e \in \IS_{n,t}(102,021) \)  
if and only if it is of one of the following two forms:  
\begin{align}
\label{eq:102021_1}
& e_1 \leq e_2 \leq \cdots \leq e_{m+t+1} = m, 
\quad e_{m+t+2} = e_{m+t+3} = \cdots = e_n = 0, \\[2ex]
\label{eq:102021_2}
& e_1 = \cdots = e_{m+k} = 0, 
\quad e_{m+k+1} = \cdots = e_{m+t+1} = m, 
\quad e_{m+t+2} = 0, \notag \\
& \quad e_{m+t+3}, \dots, e_n \in \{0, m\} \text{ with at least one entry equal to \( m \)},
\end{align}
for some \( 0 \leq k \leq t \).

In the first case~\eqref{eq:102021_1}, the sequences are in bijection with Dyck paths of semilength 
\( m+t+1 \) ending with the subpath \( ud^{t+1} \), which are counted by 
\[
\frac{t+1}{m+t+1}\binom{2m+t}{m}
\]
by Lemma~\ref{lem:Dyck_end}.

For the second case~\eqref{eq:102021_2}, the sequences are counted by 
\( (t+1)2^{n-t-m-2} \). Here, the factor \( t+1 \) counts the number of possible choices for \( k \), and 
\( 2^{n-t-m-2} -1\) counts the number of possible assignments for \( e_{m+t+3}, \dots, e_{n} \),
each of which can be either \( 0 \) or \( m \) excluding the case with \( e_{m+t+2}=\cdots=e_n=0 \).
Hence, the total number of inversion sequences \( e \in \IS_{n,t}(102,021) \) is given by
\[
(t+1) \left( \sum_{m=1}^{n-t-1} \frac{1}{m+t+1}\binom{2m+t}{m}+\sum_{m=1}^{n-t-2} (2^{n-t-m-2}-1)  \right),
\]
which completes the proof.
\end{proof}

%-----------------------------------------------------------------------------

\subsection{(102, 120)-avoiding inversion sequences}

Yan and Lin~\cite[Theorem~4.10]{YanLin20} showed that 
the total number of \((102,120)\)-avoiding inversion sequences of length \(n\) is given by
\begin{equation}\label{eq:YL120}
|\IS_{n}(102,120)| = 1+\sum_{i=1}^{n-1}\binom{2i}{i-1}.
\end{equation}

We first establish a lemma that characterizes a subclass of 
\( (102,120) \)-avoiding inversion sequences.

\begin{lem}\label{lem:102120}
Let \( n\geq 2 \) and \( 0\leq t \leq n-2 \), and define
\[ 
A_{n,t}:= \{ e\in \IS_{n,t}(102,120)  : e_{\max(e)+t}<\max(e)\}.
\]
Then the following statements hold:
\begin{enumerate}
\item For integers \( n \) and \( t \) with \( n \geq 2 \) and \( 1\leq t \leq n-1 \), 
\[ 
|A_{n+1,t}|=\sum_{i=t-1}^{n-2}|A_{n,i}| .
\]
\item For integers \( n \) and \( t \) with \( n\geq 2 \) and \( 0\leq t \leq n-2 \),
\( |A_{n,t}|=\binom{2n-t-3}{n-1} \).
\end{enumerate}
\end{lem}

\begin{proof}
\begin{enumerate}
\item Let \( e=(e_1, e_2, \dots, e_{n+1}) \in A_{n+1,t} \) with \( \max(e)=m>0 \).
Since \( e \) avoids the pattern \( 120 \) and \( e_{m+t}<e_{m+t+1}=m \), the sequence satisfies 
\( e_1\leq e_2 \leq \cdots \leq e_{m+t} \) and \( e_{m+t} \leq e_{j} \) for \( m+t+2 \leq j \leq n+1 \). 

We define a map
\[ 
\rho : A_{n+1,t} \rightarrow \bigcup_{a= 0}^{n-t-1} A_{n,t+a-1}
\]
by \( \rho(e):= (e_1, \dots, e_{m+t-1}, e_{m+t+1}-a, e_{m+t+2}-a, \dots, e_{n+1}-a) \), 
where \( a=e_{m+t}-e_{m+t-1} \geq 0 \).
Since \( e_{m+t+1}-a=e_{m+t+1}-e_{m+t}+e_{m+t-1}> e_{m+t-1} \),
it follows that \( \max(\rho(e))=e_{m+t+1}-a=m-a \) and \( \rank(\rho(e))=t+a-1 \). 
Moreover, \( \rho(e) \) preserves the weakly increasing property up to index \( m+t-1 \),
and the suffix beginning at \( e_{m+t+1}-a \) avoids the patterns \( 102 \) and \( 120 \), 
since this property is inherited from \( e \).
Thus, \( \rho(e) \in A_{n,t+a-1} \) for \( 0 \leq a \leq n-t-1 \). 

Conversely, given \( 1\leq t \leq n-2 \) and any \( \rho(e)=(f_1, f_2, \dots, f_n) \in A_{n,t+a-1} \), 
we can recover the original sequence \( e \in A_{n+1,t} \) as 
\[ 
(f_1, f_2, \dots, f_{k-2}, f_{k-2}+a, f_{k-1}+a, \dots , f_{n}+a ),
\] 
where  \( k=\max(\rho(e))+\rank(\rho(e))+2 \). This shows that \( \rho \) is a bijection.
\item We prove that \( |A_{n,t}|=\binom{2n-t-3}{n-1} \) for all \( 0 \leq t \leq n-2 \) by induction on \( n\geq 2 \).
The base case holds since \( |A_{2,0}|=|\{(0,1)\}|=\binom{1}{1} \). 
Now assume the statement holds for some \( n \geq 2 \); we aim to prove it for \( n+1 \).
By part (1) and the induction hypothesis, for \( 1 \leq t \leq n-1 \), we have
\begin{equation}\label{eq:A_n+1_t}
|A_{n+1,t}|=\sum_{i=t-1}^{n-2}|A_{n,i}|=\sum_{i=t-1}^{n-2}\binom{2n-i-3}{n-1}=\binom{2n-t-1}{n}
=\binom{2(n+1)-t-3}{(n+1)-1}.
\end{equation}
It remains to compute \( |A_{n+1,0}| \). 
To compute \( |A_{n+1,0}| \), observe that any sequence \( e \in \IS_{n+1}(102,120) \) 
satisfying \( e_{\max(e)+\rank(e)} = \max(e) \) corresponds bijectively to a sequence 
in \( \IS_n(102,120) \) by simply removing the entry \( e_{\max(e)+\rank(e)} \), and vice versa.
Thus,
\[
|\{ e \in \IS_{n+1}(102,120) : e_{\max(e)+\rank(e)}=\max(e)\} |=|\IS_{n}(102,120)|.
\]
Hence, the number of sequences in \( \IS_{n+1}(102,120) \) with \( e_{\max(e)+\rank(e)}<\max(e) \) 
is given by
\[
|\IS_{n+1}(102,120)|-|\IS_{n}(102,120)|=\left(1+\sum_{i=1}^{n}\binom{2i}{i-1} \right)
-\left( 1+\sum_{i=1}^{n-1}\binom{2i}{i-1}\right) =\binom{2n}{n-1}
\]
by \eqref{eq:YL120}.
Therefore, by \eqref{eq:A_n+1_t}, 
\[
 |A_{n+1,0}|=\binom{2n}{n-1}-\sum_{t=1}^{n-1}{|A_{n+1,t}|}
=\binom{2n}{n-1}-\sum_{t=1}^{n-1}{\binom{2n-t-1}{n}}
=\binom{2(n+1)-3}{(n+1)-1},
\]
which completes the inductive step.
\end{enumerate}
\end{proof}

\begin{prop}\label{prop:101120}
For integers \(n \geq 2\) and \(0 \leq t \leq n-2\), the number of 
\((102,120)\)-avoiding inversion sequences \(e\) of length \( n \) with \(\rank(e) = t\) is given by
\[
\binom{2n-t-2}{n-t-1}-\binom{2n-2t-3}{n-t-1}.
\]
\end{prop}

\begin{proof}
Since
\[
\sum_{i=0}^{t} \binom{2n-t-i-3}{n-t-2}=\binom{2n-t-2}{n-t-1}-\binom{2n-2t-3}{n-t-1}
\]
it suffices to show that, for each \( i \), the inversion sequences 
\( e=(e_1, e_2, \dots, e_n) \in \IS_{n,t}(102,120) \) with \( t \leq n-2 \) satisfying 
\[ 
e_1, e_2, \dots, e_{\max(e)+t-i}<e_{\max(e)+t-i+1}=\cdots = e_{\max(e)+t+1}=\max(e) 
\] 
is counted by \( \binom{2n-t-i-3}{n-t-2} \). Clearly, these sequences are in bijection with 
the \( (102,120) \)-avoiding inversion sequences in \( A_{n-i, t-i} \) by removing \( i \) consecutive \( \max(e) \) elements.
By Lemma~\ref{lem:102120}~(2), we obtain
\[ 
|A_{n-i, t-i}|=\binom{2(n-i)-(t-i)-3}{(n-i)-1}=\binom{2n-t-i-3}{n-t-2}. 
\] 
This completes the proof. 
\end{proof}

%-----------------------------------------------------------------------------

\subsection{(102, 201)-avoiding inversion sequences}

First, we enumerate \( (102, 201, 101) \)-avoiding inversion sequences \( e \) 
with \( \rank(e)=t \) and \( \max(e)=m \).

\begin{prop}\label{prop:101102201rank}
Let  \( m \) and \( n \) be positive integers. 
For \( 0\leq t \leq n-2 \) and \( 1\leq m \leq n-t-1 \), the number of inversion sequences \( e \in \IS_{n}(102,201,101) \) 
with \( \rank(e)=t \) and \( \max(e)=m \) is given by
\[
\frac{t+1}{m+t+1} \binom{2m+t}{m}\binom{n-t-2}{m-1}. 
\]
\end{prop}

\begin{proof}
Consider \( (102,201,101) \)-avoiding inversion sequences \( e=(e_1, \dots, e_n) \).
Note that \( e \) contains no triple \( 1\leq i<j<k \leq n \) 
such that \( e_{i}>e_{j}<e_{k} \). Thus, for \( 0\leq t \leq n-2 \) and \( 1\leq m \leq n-t-1 \),
\( e \in \IS_{n}(102,201,101) \) with \( \rank(e)=t \) and \( \max(e)=m \) if and only if 
\( e \) is of the form 
\begin{equation}\label{eq:102,201,101}
e_1\leq e_2 \leq \cdots \leq e_{m+t+1} =m >e_{m+t+2} \geq \cdots \geq e_{n}.
\end{equation}
Since the possible subsequences \( (e_1,\dots,e_{m+t+1}) \) are in bijection with Dyck paths of semilength \( m+t+1 \)
ending with \( ud^{t+1} \), there are 
\[ 
\frac{t+1}{m+t+1}\binom{2m+t}{m}
\]
such subsequences by Lemma~\ref{lem:Dyck_end}.
The proof is followed since there are \( \binom{n-t-2}{m-1} \) ways to make the subsequence \( (e_{m+t+2},\dots, e_{n}) \)
such that \( m>e_{m+t+2} \geq \cdots \geq e_n \geq 0 \).
\end{proof}

Now we investigate \( (102, 201) \)-avoiding inversion sequences of length \( n \) that contain the pattern \( 101 \).
For an inversion sequence \( e \), 
let \( \hat{e} \) denote the sequence obtained from \( e \) by removing all the entries \( e_i=\max(e) \).
We denote by \( \IS_{n}(102,201:101) \) the set of inversion sequences of length \( n \) 
that avoid both patterns \( 102 \) and \( 201 \), but contain the pattern \( 101 \).

\begin{lem}\label{lem:102201_101}
Let \( m \), \( n \), and \( t \) be integers such that \( n\geq 4\), \( 0\leq t \leq n-4 \), and \( 1\leq m \leq n-t-3 \).
For inversion sequences \( e=(e_1,e_2, \dots, e_n ) \), 
\( e \in \IS_{n}(102,201) \) with \( \rank(e)=t \) and \( \max(e)=m \)
contains the pattern \( 101 \) if and only if \( e \) is of the form
\begin{multline}\label{eq:102201_101}
\qquad \qquad \qquad e_1\leq e_2 \leq \cdots \leq e_{m+s} \leq \hat{m}, \quad e_{m+s+1}, \dots, e_{m+t+1}=m, \quad e_{m+t+2}=\hat{m},\\ 
e_{m+t+3}, \dots, e_{m+t+k+3} \in \{\hat{m},m\},
\quad \hat{m}>e_{m+t+k+4} \geq \cdots \geq e_{n}, \qquad \qquad \qquad \qquad
\end{multline}
with \( (e_{m+t+3}, \dots, e_{m+t+k+3})\neq (\hat{m}, \dots, \hat{m}) \) 
for some integers \( k \), \( \hat{m} \), and \( s \)
such that \( 0 \leq \hat{m} \leq m-1 \), \( 0 \leq s \leq t \), and \( 0\leq k \leq n-m-t-3 \). 

Consequently, \( \hat{e} \) is a \( (102, 201, 101) \)-avoiding inversion sequence.
\end{lem}

\begin{proof}
Consider \( (102,201) \)-avoiding inversion sequences \( e=(e_1,e_2,\dots,e_n) \).
Clearly, if \( e \) is of the form \eqref{eq:102201_101}, then \( e\in \IS_{n}(102,201:101) \) 
with \( \rank(e)=t \) and \( \max(e)=m \).

Now suppose that \( e\in \IS_{n}(102,201) \) with \( \rank(e)=t \) and \( \max(e)=m \) contains 
the pattern \( 101 \), i.e., it satisfies that \( e_i=e_k>e_j \) for some \( i<j<k \). 
If \( e_i<m \), then \( e \) has the subsequences \( (m, e_j, e_k) \) or \( (e_i, e_j, m) \),
which contradicts that \( e \) avoids both patterns \( 102 \) and \( 201 \). Hence, \( e_i=e_k=m \), 
and it follows that \( e_j=\hat{m} \), where \( \hat{m}=\max(\hat{e}) \). 
This yields that \( e \) contains the subsequence \( (m, \hat{m}, m ) \) 
and \( \hat{e} \) is a \( (102, 201, 101) \)-avoiding inversion sequence. 
Then, as in the proof of Proposition~\ref{prop:101102201rank}, \( \hat{e} \) is of the form described in \eqref{eq:102,201,101}, namely,
\[
\hat{e}_1\leq \hat{e}_2 \leq \cdots \leq \hat{e}_{r} = \hat{m} >\hat{e}_{r+1}\geq \cdots \geq \hat{e}_{\ell}
\]
for some \( r \) and \( \ell \). 
Since \( e \) avoids both patterns \( 102 \) and \( 201\) and having \( \rank(e)=t \) and \( \max(e)=m \), 
\( e \) needs to satisfy \eqref{eq:102201_101}.
Lastly, \( (e_{m+t+3}, \dots, e_{m+t+k+3})\neq (\hat{m}, \dots, \hat{m}) \) 
since \( e \) contains \( (m,\hat{m}, m) \) as a subsequence. 
This completes the proof.
\end{proof}

\begin{prop}\label{prop:102201_101}
Let \( m \), \( n \), and \( t \) be integers such that \( n\geq 4 \), \( 0\leq t \leq n-2 \), and \( 1\leq m\leq n-t-3 \).
The number of inversion sequences \( e \in \IS_{n}(102,201:101) \) with \( \rank(e)=t \) and \( \max(e)=m \) is given by
\[
(t+1)\big(2^{n-m-t-2}-1\big) \\
\quad + \sum_{j=1}^{m-1} \sum_{s=0}^{t} \sum_{k=0}^{n-m-t-3} 
(2^{k+1}-1) \cdot \frac{m+s-j+1}{m+s+1} \binom{m+j+s}{j} 
\binom{n+j-m-t-k-4}{j-1}.
\]
\end{prop}

\begin{proof}
By Lemma~\ref{lem:102201_101}, we need to count the inversion sequences 
of the form \eqref{eq:102201_101} with the additional property \( (e_{m+t+3}, \dots, e_{m+t+k+3})\neq (j, \dots, j) \), for some integers \( k \), \( j \), and \( s \),
such that \( 0 \leq j \leq m-1 \), \( 0 \leq s \leq t \), and \( 0\leq k \leq n-m-t-3 \). 

In the special case \( j=0 \), the entries of \( e \) are restricted to the set \( \{ 0, m \} \). 
In this case, it should be \( k = n-m-t-3 \), so there are \( (t+1)(2^{n-m-t-2}-1) \) inversion sequences of the form 
\begin{align*}
&e_1=\cdots=e_{m+s}=0, \quad e_{m+s+1}=\cdots=e_{m+t+1}=m, \quad e_{m+t+2}=0, \\
&e_{m+t+3}, \dots, e_{n}\in \{0,m\} \text{~with at least one entry equal to \( m \)}
\end{align*}
with \( 0\leq s \leq t \). 

Now consider the case where \( j > 0 \). 
The possible weakly increasing subsequences \( (e_1, \dots, e_{m+s}) \) satisfying \( e_{m+s} \leq j \) 
are in bijection with Dyck paths of semilength \( m+s+1 \) ending with \( ud^{m+s-j+1} \). 
Hence, by Lemma~\ref{lem:Dyck_end}, their count is
\[
\frac{m+s-j+1}{m+s+1}\binom{m+j+s}{j}.
\]
The subsequence \( (e_{m+t+3}, \dots, e_{m+t+k+3}) \in \{j, m\}^{k+1} \setminus \{j\}^{k+1} \) 
contributes \( 2^{k+1} - 1 \) choices.  
Finally, the number of weakly decreasing sequences \( (e_{m+t+k+4}, \dots, e_n) \) with entries 
at most \( j-1 \)
is given by \( \binom{n+j-m-t-k-4}{j-1} \).  
Combining these counts completes the proof.
\end{proof}

By summing over both classes of sequences—those avoiding and those containing the pattern 
\( 101 \)—we obtain the total count of 
\( (102,201) \)-avoiding inversion sequences of given rank.

\begin{cor}\label{cor:102201_total}
For integers \(n \geq 2\) and \(0 \leq t \leq n-2\), the number of 
\((102,201)\)-avoiding inversion sequences \(e\) of length \( n \) with \(\rank(e) = t\) is given by
\[
\sum_{m=1}^{n-t-1} a(n,t,m) + \sum_{m=1}^{n-t-3} b(n,t,m),
\]
where
\begin{align*}
a(n,t,m) &= \frac{t+1}{m+t+1} \binom{2m+t}{m} \binom{n-t-2}{m-1}, \\
b(n,t,m) &= (t+1)\big(2^{n-m-t-2}-1\big) \\
&\quad + \sum_{j=1}^{m-1} \sum_{s=0}^{t} \sum_{k=0}^{n-m-t-3} 
(2^{k+1}-1) \cdot \frac{m+s-j+1}{m+s+1} \binom{m+j+s}{j} 
\binom{n+j-m-t-k-4}{j-1}.
\end{align*}
\end{cor}

%------------------------------------------------------------------------

\subsection{(102, 210)-avoiding  inversion sequences}
In this subsection, we identify the labeled $F$-paths corresponding to 
\( (102,210) \)-avoiding inversion sequences with a fixed rank, and subsequently enumerate these paths.

For a nonnegative integer $n$, recall \(\LFp{n} \) denotes the set of labeled \( F \)-paths of semilength \( n \). Let \(\LFp{n,t} \) denote the set of $Q$ in $\LFp{n}$ with $\height(Q)=t$.
For each step $(a,b)$ in a labeled $F$-path $Q$, we call a step $(a,b)$ 
\begin{itemize}
\item \emph{north} if $a=0$ and $b=1$,
\item \emph{up} if $a\ge1$ and $b=1$,
\item \emph{down} if $a\ge1$ and $b\le0$.
\end{itemize}
Now, consider a down step $(a,b)$ with an associated label $(a;b_1,\ldots,b_k)$. We call a down step $(a,b)$  
\begin{itemize}
\item \emph{pure} if $k=1$,
\item \emph{$0$-tailed} if $k\ge2$ and $b_2=\dots=b_k=0$, 
\item \emph{complex} if \( k\ge2 \) and there exists at least one \( i \in \{ 2, \dots, k \} \) such that \( b_i<0 \).
\end{itemize}

To analyze a particular subclass of labeled 
\( F \)-paths related to \( (102,210) \)-avoiding inversion sequences, 
we define two disjoint subsets of \( \LFp{n,t} \) based on the nature and placement of the (at most one) down step:
Let $\LFa_{n,t}$ be the set of $Q\in \LFp{n,t}$ such that $Q$ contains at most one down step, which (if present) 
is 
either pure or 0-tailed; and let $\LFb_{n,t}$ be the set of $Q\in \LFp{n,t}$ such that $Q$
contains exactly one down step, which is complex, and all steps following it (if any) are north steps.
Their union defines the subclass of interest:
 
$$\LFp{n,t}(210):=\LFa_{n,t}\cup\LFb_{n,t}.$$

\begin{lem}\label{lem:102210}
For nonnegative integers \( n \) and \( t \), we have 
$$\phi(\LFp{n,t}(210))=\IS_{n+1,t}(102,210).$$ 
\end{lem}

\begin{proof}
First, suppose \( Q \in \LFa_{n,t} \), meaning that it contains at most one down step, which is either pure or 0-tailed.  
Then \( \phi(Q)  \) has at most one descent, and 
hence it avoids the \(210 \) pattern.
Therefore, \( \phi(Q) \in \IS_{n+1,t}(102,210) \).  

Next, suppose \( Q \in \LFb_{n,t} \), so that \( Q \) has a unique down step which is complex, and all subsequent steps (if any) are north steps.  
These north steps cannot contribute to a \( 210 \) pattern in \( \phi(Q) \), and 
although the unique complex down step may introduce one or more descents due to the insertion 
of a maximum value, the resulting sequence still avoids the 210 pattern.
Therefore, \( \phi(Q) \in \IS_{n+1,t}(102,210) \) in this case as well.  

On the other hand, suppose \( Q \notin \LFp{n,t}(210) \).  
If \( Q \) contains two or more down steps, then by construction \( \phi(Q) \) contains the pattern \( 210 \).  
If \( Q \) has a unique down step that is complex and followed by an up step, then again \( \phi(Q) \) contains \( 210 \).  
Hence, in both cases, \( \phi(Q) \notin \IS_{n+1,t}(102,210) \).

Combining the above observations, we conclude that
\[
\phi(\LFp{n,t}(210)) = \IS_{n+1,t}(102,210).
\]
\end{proof}

For a nonnegative integer $t$, let
\begin{align*}
A_t=A_t(x)&:=\sum_{n\ge 0} \left|\LFa_{n,t}\right| x^n,\\
A(u,x)&:=\sum_{n, t\ge 0} \left|\LFa_{n,t}\right| u^t x^n = \sum_{t\ge 0} A_t(x)\, u^t, \\
B(u,x)&:=\sum_{n, t\ge 0} \left|\LFb_{n,t}\right| u^t x^n.
\end{align*}
By decomposing the labeled $F$-paths in $\cup_{n\ge1}\LFa_{n,t}$, we have
\begin{equation}\label{eq:At-A0}
A_t = x^t C^{t+1} + (t+1)\left(A_0-C\right)x^t C^t=\left( (t+1)A_0-tC\right)(xC)^t,
\end{equation}
where \( C=C(x) \) denotes the generating function for the Catalan numbers \( C_n \), 
which counts the number of labeled $F$-paths in $\LFa_{n,0}$ of semilength $n$, 
consisting only of north and up steps.

To find the generating function \( A_0 \), let us consider
the last step of \( Q \in \LFa_{n,0} \) with \( n \ge 1 \).

\begin{itemize}
  \item If the last step is an up step of the form \( (t+1,1) \), then this case contributes \( x \cdot A_t \).
  \item If the last step is a pure down step \( (a, a-t) \)  for \( 1 \le a \le t \), then this case contributes \( tx \cdot x^t C^{t+1} \) since the last step is the only down step. Here,  the factor \( t \) accounts for the number of choices for \( a \).
  \item If the last step is a $0$-tailed down step  \( (a, a-t) \) labeled as \( (a; a-t, 0, \ldots, 0) \) for \( 1 \le a \le t \),  
  then this case contributes
  \[
  tx \cdot (x + x^2 + x^3 + \cdots) \cdot x^t C^{t+1},
  \]
  since the last step is again the only down step, and the length of its $0$-tail can be any positive integer.  
 Here, the generating function \( x + x^2 + x^3 + \cdots \) accounts for the number of ways to assign a positive
  length $0$-tail.
\end{itemize}
Thus, we get 
the following equation for \( A_0 \):

\begin{align*}
A_0-1
&=\sum_{t\ge0}xA_t+\sum_{t\ge1}t(xC)^{t+1}+\sum_{t\ge1}t\frac{x}{1-x}(xC)^{t+1} \\
&=xA_0\sum_{t\ge0}(t+1)(xC)^t+\frac{x}{1-x}\sum_{t\ge1}t(xC)^{t+1} \tag{$\because$~\eqref{eq:At-A0}}\\
&=\frac{xA_0}{(1-xC)^2}+\frac{x}{1-x}\frac{x^2 C^2}{(1-xC)^2}\\ 
&=xC^2 A_0+\frac{x^3 C^4}{1-x} \tag{$\because$ $C=\frac{1}{1-xC}$}.
\end{align*}
Solving the equation for $A_0$, we have
\begin{align*}
A_0
&=\frac{1+x^3 C^4/(1-x)}{1-xC^2}\\
&=C+\frac{xC^3-xC^2+x^3 C^4/(1-x)}{1-xC^2}\tag{$\because$~$1-C=-xC^2$}\\
&=C+\frac{1}{1-x}\frac{x^2 C^4}{1-xC^2}\tag{$\because$~$C^3-C^2=x C^4$}\\
&=C+\frac{1}{1-x}\frac{x^2 C^3}{\sqrt{1-4x}}\tag{$\because$~$\frac{C}{1-xC^2}=\frac{1}{\sqrt{1-4x}}$}.
\end{align*}
Substituting this expression for \( A_0 \)
into the identity in~\eqref{eq:At-A0}, we obtain the following closed form for \( A(u,x) \).
\begin{align*}
A(u,x)
&=\sum_{t\ge0}\left( (t+1)A_0-tC\right)(uxC)^t\\
&=\frac{A_0-uxC^2}{\left(1-uxC\right)^2}\\
&=\frac{x^2C^3}{(1-uxC)^2(1-x)\sqrt{1-4x}}+\frac{C}{1-uxC}.
\end{align*}

We now turn to the generating function \( B(u,x) \), which enumerates paths in \( \LFb_{n,t} \).  
For a nonnegative integer $t$, let $\wtd{\LFb}_{n,t}$ denote the 
subset of ${\LFb}_{n,t}$
consisting of paths whose unique down step appears at the end and is a complex down step. 
Define
\begin{align*}
\wtd{B}_t=\wtd{B}_t(x)&:=\sum_{n\ge 0} \left|\wtd{\LFb}_{n,t}\right| x^n,\\
\wtd{B}(u,x)&:=\sum_{n, t\ge 0} \left|\wtd{\LFb}_{n,t}\right| u^t x^n= \sum_{t\ge 0} \wtd{B}_t(x)\, u^t.
\end{align*}
By decomposing the labeled $F$-paths in $\cup_{n\ge1}\wtd{\LFb}_{n,t}$, we obtain the relation 
\begin{equation}\label{eq:B_t}
\wtd{B}_t = (xC)^t \wtd{B}_0, 
\end{equation}
where \( C=C(x) \) denotes the generating function for the Catalan numbers \( C_n \), 
which counts the number of labeled $F$-paths in ${\LFa}_{n,0}$ of semilength $n$, 
consisting only of north and up steps.

To find the generating function \( \wtd{B}_0 \), let us examine
the last step $(a,b)$ of a labeled $F$-path \( Q \in \wtd{\LFb}_{n,0} \). For each value $s=a-b$,
the last step is a complex down step $(a, a-s)$ labeled as $(a;b_1,\ldots,b_k)$, where
$$ 1\le a \le s-1,  \quad k\ge 2,\quad b_1,\ldots,b_k \le 0, \quad b_1+\cdots+b_k=a-s, \quad (b_2,\ldots, b_k)\neq(0,\ldots,0).  $$
Thus, the contribution corresponding to a fixed value of $s$ is
$$
\sum_{k\ge2}\sum_{a=1}^{s-1}\left( \binom{s+k-a-1}{k-1}-1\right)x^k\cdot x^s C^{s+1}=\left(\frac{1}{(1-x)^s}-\frac{sx}{1-x}-1 \right)x^s C^{s+1}.
$$
Therefore, we get 
the following equation for \( \wtd{B}_0 \):
\begin{align*}
\wtd{B}_0&=\sum_{s\ge2}\left(\frac{1}{(1-x)^s}-\frac{sx}{1-x}-1 \right)x^s C^{s+1}.\\
&=\frac{x^2C^5}{1-x}-\frac{x^2 C^2}{1-x}\left(C^2-1 \right)-{x^2 C^4}\\
&=\frac{x^4C^7}{1-x}.
\end{align*}
The simplification can be carried out as in the computation of $A_0$, 
using similar generating function manipulations.
Due to \eqref{eq:B_t}, we have
$$\wtd{B}(u,x)=\sum_{t\ge0} \wtd{B}_t u^t = \frac{\wtd{B}_0}{1-uxC}. $$
Since all steps of $Q\in \LFb_{n,t}$ following the complex down step (if any) are north steps,
each such sequence contributes an additional factor of $\frac{1}{1-ux}$.
Therefore, we obtain

\begin{align*}
B(u,x)=\wtd{B}(u,x)\cdot\frac{1}{1-ux}&=\frac{x^4 C^7}{(1-x)(1-uxC)(1-ux)}\\
&=\frac{x^3C^5}{1-x}\left(\frac{C}{1-uxC}-\frac{1}{1-ux} \right).
\end{align*}

Combining the two types of labeled \( F \)-paths, we define the total generating function
$$G(u,x):=\sum_{n, t\ge 0} \left|\LFp{n,t}(210)\right| u^t x^n.$$
Then,
\[
G(u,x)=A(u,x)+{B(u,x)},\\
\]
and by substituting the expressions derived above, we obtain:
\[
G(u,x)=\frac{x^2C^3/(1-x)}{\sqrt{1-4x}}\frac{1}{(1-xCu)^2}+\frac{C}{1-xCu}+\frac{x^3 C^6/(1-x)}{1-xCu}-\frac{x^3 C^5/(1-x)}{1-xu}.
\]
Thus, we obtain
\begin{equation*}
\left[u^t\right] G(u,x)=\frac{(t+1)x^{t+2}C^{t+3}}{(1-x)\sqrt{1-4x}}+x^t C^{t+1} + \frac{x^{t+3}C^{t+6}}{1-x}-\frac{x^{t+3}C^{5}}{1-x}.
\end{equation*}
Since 
\[
\left[x^{n-1}\right]\left(\left[u^t\right] G(u,x)\right)=\left|\LFp{n-1,t}(210)\right|
= \left|\IS_{n,t}(102,210)\right|
\] 
for $n\ge 1$, we arrive at the following enumerative result:

\begin{prop}
For integers \(n \geq 2\) and \(0 \leq t \leq n-1\), the number of 
\((102,210)\)-avoiding inversion sequences \(e\) of length \( n \) with \(\rank(e) = t\) is given by
$$
c(n-t-1,t+1) + (t+1)\sum_{i=0}^{n-t-3}\binom{2i+t+3}{i}+\sum_{i=0}^{n-t-4}\left(c(i,t+6)-c(i,5) \right),
$$
where $c(j,k):=\frac{k}{2j+k}\binom{2j+k}{n}=[x^j]C^k$.
\end{prop}

%-------------------------------------------------------------------------

\subsection{(102, 110)-avoiding  inversion sequences}
In this subsection, we aim to enumerate 
\( (102,110) \)-avoiding inversion sequences by employing a similar generating function approach as in the previous case of \( (102,210) \)-avoidance.

Let $\LFp{n,t}(110)$ be the set of $Q\in \LFp{n,t}$ satisfying the following conditions:

\begin{itemize}
\item \( Q \) contains no complex down steps;
\item every \( 0 \)-tailed down step in \( Q \) is preceded only by north or up steps;
\item no north step occurs after any down step.
\end{itemize}

\begin{lem}\label{lem:102110}
For nonnegative integers \( n \) and \( t \), we have 
$$\phi(\LFp{n,t}(110))=\IS_{n+1,t}(102,110).$$ 
\end{lem}

\begin{proof}
Let \( Q \in \LFp{n,t} \). 
Suppose \( Q \) contains a down step labeled \( (a; b_1, \ldots, b_k) \) with \( k \ge 2 \). 
Since \( k \ge 2 \), the current maximum value \( m \) assigned at the moment of insertion is inserted at least twice in \( \phi(Q) \). 
Note that if \( b_i < 0 \) for some \( i \ge 2 \), then the insertion at that position results in a smaller value following a repeated entry \( m \), thereby creating a descent. 
This produces an instance of the pattern \( 110 \) in \( \phi(Q) \). 
Hence, if \( \phi(Q) \in \IS_{n+1,t}(102,110) \), then \( Q \) must not contain any complex down steps.

Next, suppose \( Q = s_1 \dots s_n \) contains a 0-tailed down step \( s_i \), labeled \( (a; b_1, \ldots, b_k) \) with \( k \ge 2 \). 
Assume that \( \phi(s_1 \dots s_{i-1}) \) has already been constructed, and let \( r \) be the current maximum value at that stage. 
The step \( s_i \) inserts a new value \( m > r \) exactly \( k \) times into the sequence. 
More precisely, one occurrence of \( m \) is inserted to the left of \( r \), and the remaining \( m \)'s are placed consecutively to the right of \( r \).
Note that if all of the steps \( s_1, \ldots, s_{i-1} \) are north or up steps, then \( \phi(s_1 \dots s_{i-1}) \) contains no descent. 
In this case, the insertion of the repeated \( m \)'s around \( r \) does not produce a \( 110 \) pattern.  
However, if there exists at least one down step among \( s_1, \dots, s_{i-1} \), then \( \phi(s_1 \dots s_{i-1}) \) already contains a descent, and together with the inserted \( m \)'s, this results in an occurrence of the pattern \( 110 \).

Finally, if a north step appears after a down step,
then this necessarily produces the pattern \( 110 \) in \( \phi(Q) \). Hence, such cases must also be excluded.

These conditions precisely characterize the set of paths \( Q \in \LFp{n,t} \) such that \( \phi(Q) \in \IS_{n+1,t}(102,110) \). 

Conversely, suppose \( Q \in \LFp{n,t}(110) \). 
By definition, \( Q \) contains no complex down steps, no north step is followed by a down step, and any 0-tailed down step is preceded only by up or north steps. 
As shown above, each of these conditions guarantees that \( \phi(Q) \) avoids the pattern \( 110 \). 
Since \( Q \in \LFp{n,t} \), we also have \( \phi(Q) \in \IS_{n+1,t}(102) \), and hence \( \phi(Q) \in \IS_{n+1,t}(102,110) \).

Thus, we conclude that
\[
\phi(\LFp{n,t}(110)) = \IS_{n+1,t}(102,110).
\]
\end{proof}

We now turn to the enumeration of such paths using generating functions.
Let
\begin{align*}
H_t=H_t(x)&:=\sum_{n\ge 0} \left|\LFp{n,t}(110)\right| x^n,\\
H(u,x)&:=\sum_{n, t\ge 0} \left|\LFp{n,t}(110)\right| u^t x^n = \sum_{t\ge0} H_t \,u^t.
\end{align*}
Recall \( C = C(x) \) denotes the generating function for the Catalan numbers \( C_n \), and in this context, \( C_n \) counts the number of labeled \( F \)-paths in \( \LFp{n,0}(110) \) of semilength \( n \), consisting only of north and up steps.
Since no north step is followed by a down step, decomposing labeled $F$-paths in $\cup_{n\ge1}\LFp{n,t}(110)$ induce that

\begin{align} \label{eq:decomp110}
H_t&=(xC)^t H_0,\\
H(u,x)&=\frac{H_0}{1-uxC}\notag.
\end{align}

To find the generating function \( H_0 \), 
let us consider
the last step of \( Q \in \LFp{n,0}(110) \) with \( n \ge 1 \).

\begin{itemize}
  \item If the last step is an up step of the form \( (t+1,1) \), it contributes \( x \cdot H_t \).
  \item If the last step is a pure down step \( (a, a-t) \) for \( 1 \le a \le t \), it contributes \( tx \cdot H_t \). As before, the factor \( t \) accounts for the number of choices for \( a \).
  \item If the last step is a $0$-tailed down step \( (a, a-t) \)  for \( 1 \le a \le t \), labeled as \( (a; a-t, 0, \ldots, 0) \), it contributes
\[
tx \cdot (x + x^2 + x^3 + \cdots) \cdot x^t C^{t+1},
\]
since the last step is
the only down step.
\end{itemize}
Thus, we get
\begin{align*}
H_0-1&=\sum_{t\geq0}xH_t+\sum_{t\geq1}t x H_t+\sum_{t\ge1}t\,(x^2+x^3+x^4+\cdots) x^t C^{t+1}\\
&=xH_0\sum_{t\geq0}(t+1)(xC)^t+\frac{x^2 C}{1-x}\sum_{t\ge1}t (xC)^t \tag{$\because$ \eqref{eq:decomp110}}\\
&=\frac{xH_0}{(1-xC)^2} + \frac{x^3C^2}{(1-x)(1-xC)^2}\\
&=xC^2 H_0 + \frac{x^3 C^4}{1-x},\tag{$\because$ $C=\frac{1}{1-xC}$}
\end{align*}
which yields that
\begin{align*}
H_0&=\frac{1}{1-xC^2}\left(1+\frac{x^3 C^4}{1-x} \right) \\
&= \frac{1-x+x(C-1)^2}{(1-xC^2)(1-x)} \tag{$\because$ $xC^2=C-1$}\\
&= \frac{1-2x}{(1-x)\sqrt{1-4x}}. \tag{$\because$ $\frac{C}{1-xC^2}=\frac{1}{\sqrt{1-4x}}$}
\end{align*}
From $H(u,x)=\frac{H_0}{1-xCu}$, we have
\begin{align*}
[u^t]H(u,x)&=H_0\left(xC\right)^t\\
&=\frac{1-2x}{1-x}\frac{\left(xC\right)^t}{\sqrt{1-4x}}\\
&=\left(1-x-x^2-x^3-\cdots\right)\sum_{n\ge t} \binom{2n-t}{n-t}x^n.
\end{align*}

Since 
\[
\left[x^{n-1}\right]\left(\left[u^t\right] H(u,x)\right)=\left|\LFp{n-1,t}(110)\right|=|\IS_{n,t}(102,110)|,
\]
we have the following proposition.

\begin{prop}
For integers \(n \geq 2\) and \(0 \leq t \leq n-1\), the number of 
\((102,110)\)-avoiding inversion sequences \(e\) of length \( n \) with \(\rank(e) = t\) is given by
$$
\binom{2n-t-2}{n-t-1}-\sum_{i=2}^{n-t}\binom{2n-t-2i}{n-t-i}.
$$
\end{prop}

%%%%%%%%%%%%%%%%

%\section*{Acknowledgements}

%--------------------------------------------------------------------------------
%\nocite{*}
%\bibliographystyle{abbrvnat}
% use the following instead if you encounter problems

%\bibliographystyle{alpha}
%%\bibliographystyle{amsabbrv}
%\bibliography{\jobname}

\begin{thebibliography}{10}

\bibitem[CMSW16]{CMSW16}
Sylvie Corteel, Megan~A. Martinez, Carla~D. Savage, and Michael Weselcouch.
\newblock Patterns in inversion sequences {I}.
\newblock {\em Discrete Math. Theor. Comput. Sci.}, 18(2):Paper No. 2, 21,
  2016.

\bibitem[Ges16]{Ges16}
Ira~M. Gessel.
\newblock Lagrange inversion.
\newblock {\em J. Combin. Theory Ser. A}, 144:212--249, 2016.

\bibitem[HKSS24]{HKSS24}
JiSun Huh, Sangwook Kim, Seunghyun Seo, and Heesung Shin.
\newblock Bijections on pattern avoiding inversion sequences and related
  objects.
\newblock {\em Adv. in Appl. Math.}, 161:Paper No. 102771, 40, 2024.

\bibitem[HL22]{HongLi22}
Letong Hong and Rupert Li.
\newblock Length-four pattern avoidance in inversion sequences.
\newblock {\em Electron. J. Combin.}, 29(4):Paper No. 4.37, 15, 2022.

\bibitem[Kit11]{Kitaev11}
Sergey Kitaev.
\newblock {\em Patterns in permutations and words}.
\newblock Monographs in Theoretical Computer Science. An EATCS Series.
  Springer, Heidelberg, 2011.
\newblock With a foreword by Jeffrey B. Remmel.

\bibitem[KMY24]{KMY24}
Ilias Kotsireas, Toufik Mansour, and G\"{o}khan Y{\i}ld{\i}r{\i}m.
\newblock An algorithmic approach based on generating trees for enumerating
  pattern-avoiding inversion sequences.
\newblock {\em J. Symbolic Comput.}, 120:Paper No. 102231, 18, 2024.

\bibitem[MS15]{MS15}
Toufik Mansour and Mark Shattuck.
\newblock Pattern avoidance in inversion sequences.
\newblock {\em Pure Math. Appl. (PU.M.A.)}, 25(2):157--176, 2015.

\bibitem[MS18]{MartinezSavage18}
Megan Martinez and Carla Savage.
\newblock Patterns in inversion sequences {II}: inversion sequences avoiding
  triples of relations.
\newblock {\em J. Integer Seq.}, 21(2):Art. 18.2.2, 44, 2018.

\bibitem[{OEI}25]{OEIS}
{OEIS Foundation Inc.}
\newblock The {O}n-{L}ine {E}ncyclopedia of {I}nteger {S}equences, 2025.
\newblock Published electronically at \url{http://oeis.org}.

\bibitem[SS23]{SS23}
Seunghyun Seo and Heesung Shin.
\newblock On {D}elannoy paths without peaks and valleys.
\newblock {\em Discrete Math.}, 346(7):Paper No. 113399, 12, 2023.

\bibitem[Tes24]{Testart24}
Benjamin Testart.
\newblock Completing the enumeration of inversion sequences avoiding one or two
  patterns of length 3, 
\newblock  arXiv:2407.07701v3, 2024.

\bibitem[YL21]{YanLin20}
Chunyan Yan and Zhicong Lin.
\newblock Inversion sequences avoiding pairs of patterns.
\newblock {\em Discrete Math. Theor. Comput. Sci.}, 22(1):Paper No. 23, 35,
  [2020--2021].

\end{thebibliography}

\end{document}